\documentclass[11pt,reqno,twoside]{article}
\usepackage{graphicx}%
\usepackage{multirow}%
\usepackage{amsmath,amssymb,amsfonts}%
\usepackage{amsthm}%
\usepackage{mathrsfs}%
\usepackage[title]{appendix}%
\usepackage[dvipsnames]{xcolor}%
\usepackage{textcomp}%
\usepackage{manyfoot}%
\usepackage{booktabs}%
\usepackage{algorithm}%
\usepackage{algorithmicx}%
\usepackage{algpseudocode}%
\usepackage{listings}%
\usepackage{fullpage}
\usepackage{ulem}
\usepackage{cancel}
\makeatletter
\newcommand{\mathsout}[1]{\mathpalette\mathsout@{#1}}
\newcommand{\mathsout@}[2]{%
	\sbox0{$#1#2$}%
	\ooalign{%
		\copy0\cr
		\hidewidth\rule[0.5ex]{\wd0}{0.4pt}\hidewidth\cr
	}%
}
\makeatother

\usepackage{mysty}
\newcommand{\mmse}{{\mathrm{mmse}}}

\title{Physics Matters in PnP: Recovery Guarantees with the MMSE and NN Denoisers}
\author{Tobias Wolf\thanks{UNICAEN, CNRS, GREYC, Caen, France (tobias.wolf@unicaen.fr)} 
\and Jalal Fadili\thanks{ENSICAEN, CNRS, GREYC, Caen, France (Jalal.Fadili@ensicaen.fr)}
\and Jin Guo\thanks{Department of Mathematics, City University of Hong Kong, Hong Kong(jin.guo@my.cityu.edu.hk)}
\and Roy Y. He\thanks{Department of Mathematics, City University of Hong Kong, Hong Kong (royhe2@cityu.edu.hk)} }
\date{\today}
\begin{document}
	\maketitle
	
	\begin{abstract}
		We investigate the forward-backward-splitting version of the Plug and Play (PnP) method for linear ill-posed problems with MMSE estimators as denoisers. In contrast to existing literature, we consider estimators which are specialized for (degenerate) Gaussian noise with possibly non-diagonal covariance matrices. We further deviate from the classical iteration by replacing parts of the descent step with a linear operator that relates the observation noise to that of the MMSE estimator. Under mild assumptions, we derive several properties of the denoiser and prove recovery guarantees of the iteration both pointwise and in the Wasserstein distance of the underlying probability distributions. Crucially, our analysis shows that the denoiser cannot be chosen in a physics-agnostic way, that is, independently of the forward model. We extend our results to the case where the MMSE denoiser is parametrized by a neural network and derive the corresponding recovery bounds.
	\end{abstract}
	
	\tableofcontents
	\allowdisplaybreaks

	\section{Introduction}
	\subsection{Problem statement and literature overview}
	Let $\cX \subset \R^n$  and $A: \cX \to \R^m$ be a linear forward operator. For random vectors $\bX$ and $\bE$  distributed according to, respectively, the probability measures $\mu_\bX \in \cP(\cX)$ and $\mu_\bE \in \cP(\R^m)$, we make an observation $\by$ according to the forward model
	\begin{equation}\label{eq_data_Model}
		\by = A\bx + \be .
	\end{equation}
	Here $\bx$ and $\be$ are realizations of $\bX$ and $\bE$, respectively. Throughout this work, we assume that the noise $\be$ is Gaussian i.e. 
	\begin{assumption}\label{assum:noiseonY}
		$\mu_\bE$ is the zero-mean Gaussian distribution with positive definite covariance matrix $\Sigma_\bE$.
	\end{assumption}
	
	Solving \eqref{eq_data_Model} is, in general, an ill-conditioned problem and suffers additionally from non-uniqueness if $A$ is not injective. In order to overcome these issues, it is necessary to restrict the inversion process to a well-chosen subset of $\cX$ containing the plausible solutions, as well as to stabilize the inversion process via so called regularization methods. 
	
\subsubsection{Model-based approach}	
Many classical solutions address the problem of selecting appropriate solutions by using regularization functionals that encode some prior on the sought-after vectors. The general principle is to promote vectors with some notion of low-complexity, or equivalently, to penalize, even with a hard constraint, vectors with undesired structures as given by the model knowledge. A multitude of regularizers have been designed, among which the most popular are total variation and its extensions~\cite{rudin1992nonlinear,chambolle1997image, TGV}, as well as low-rank, sparsity and group sparsity promoting penalties. See \cite{mallat1999wavelet,scherzer2009variational,starck2010sparse,vaiter2015low} for comprehensive overviews in finite and infinite dimensions. 
In order to approximate solutions of \eqref{eq_data_Model}, these regularizers are often employed as parts of an energy functional that is minimized \cite{chambolle1997image, Donoho2006, AubertAujol2008}, within iterative methods in which a sequence of variational problems is solved \cite{orig_breg, Goldstein2009, VeseDeblurring}, or implicitly within iterative methods that aim to approximate a specific minimizer of a discrepancy functional \cite{Leitao2009, Sprung2018, QJin2023}.\\
Model-based regularization has produced a wealth of theoretical recovery guarantees for a large variety of regularizers; see \cite{engl1996regularization,scherzer2009variational,benning2018modern,vaiter2015low}. They are explainable, and the failure cases are well-known. However, the choice of the prior and thus of the regularization is nontrivial, and it might not fully capture the diversity and complexity of the objects to recover. This has led to a tremendous shift toward data-driven methods where the prior is either partly or completely learned from some training data.

\subsubsection{Data-driven approach}
In order to circumvent some of the issues discussed above when using model-based priors, it was proposed to learn  this prior information from data, rather than by handcrafting models. This can be done either partly or completely, explicitly or implicitly. This trend has started even before the advent of ``deep learning''. Our review here is by no means exhaustive in a literature that has been constantly and rapidly evolving in recent years. Observe that completely separating existing literature into model-based versus data-driven is debatable. We invite the reader to refer to the excellent and comprehensive reviews to~\cite{arridge_solving_2019,ongie_deep_2020} as well as to more recent surveys on specific themes~\cite{EldarModelReview21,monga_algorithm_2021,mukherjee2021end,kamilov2023plug}. See also the volume \cite{HandbookIP25}.
	
A natural, yet naive, way is to learn a neural network (NN) that approximates an analytic inverse to the forward operator $A$; see~\cite{jin2017deep,zhu2018image,Petersen23} and others. This approach is fully supervised but  suffers from a few shortcomings. First, it does not explicitly take into account the physics of the problem (the forward model). Attempts to overcome this issue include the work \cite{schwab2019deep}. Second, finding an appropriate NN architecture to approximate an inverse operator well is not an easy task. Moreover, this has to be done robustly to stay immune to perturbation by the noise, so that one may possibly derive stability and recovery guarantees. In fact, a continuous inverse of $A$ rarely exists which may jeopardize the whole approach. Generally, directly learning an inversion map lacks a deep understanding of its recovery guarantees, with a notable exception of the recent work of \cite{Petersen23}.

\subsubsection{Hybrid approach}
To overcome some of the shortcomings discussed for the previous learning procedure, a variety of hybrid approaches that mix model- and data-driven methods were devised with the aim of combining the best of both worlds. In the following, we review some of these.
	
\paragraph{Learning the regularization functional}
This approach consists of parametrizing the regularizing functional, typically involving a neural network, which is then learned either partly or completely; see e.g.~\cite{lunz2018adversarial,kobler2020total,prost_learning_2021,mukherjee2023learned}. The training is typically cast as a bilevel programming problem, which, generally, is very challenging to solve. Regularization guarantees were obtained in this setting in~\cite{li_nett_2020,mukherjee2020learned,schwab2019deep,Shumaylov2024}. However, this does not give guarantees about the real sought-after vector.

\paragraph{Plug-and-Play}
A now widely used approach is Plug-and-Play (PnP for short)~\cite{venkatakrishnan_plug-and-play_2013, romano2017little, kamilov2023plug, huraultthesis, TanHNA26}. It is inspired by operator splitting schemes when applied to solve variational model-based problems. These splitting schemes typically involve the so-called (explicit) proximal mapping associated to the regularizer. The idea is now to replace the proximal operator with any denoising operator, hence the name `Plug-and-Play'. Observe that this means that the regularization functional is now implicitly defined. By parameterizing it with a neural network, the denoiser is learned from pairs of noisy-noiseless data once and for all. In this case, the forward model is well captured by the data fidelity in the splitting algorithm. Recalling that the denoiser induces an implicit prior and adopting a Bayesian perspective, several authors have combined PnP with Monte-Carlo simulation, for instance discretization of Langevin diffusion, to solve inverse problems~\cite{laumont2022bayesian, guo2019agem, kadkhodaie2021stochastic, song2019generative, ho2020denoising, ehrhardt2024proximal}.

\paragraph{Unrolling/Unfolding}
The starting point of unfolding, also referred to as unrolling, is to convert an iterative optimization algorithm into a neural network such that each step (implicit or explicit) of the optimization algorithm resembles a layer of a neural network, and then replace some steps with learned ones. In a nutshell, unrolling amounts to learning parts of optimization algorithms. Its time-continuous version can be seen as an optimal control problem. The idea goes back to \cite{gregor2010learning} for the ISTA algorithm in sparse recovery. Since then, it has received a tremendous interest; see \cite{gregor2010learning,mukherjee2021end,adler2017solving,tamir2019unsupervised,bertocchi2020deep} and the reviews in \cite{monga_algorithm_2021,EldarModelReview21}. Despite its success and popularity, unrolling has a few shortcomings nevertheless. Its computational feasibility is directly tied to the number of iterations, where one can then have as many neural networks as iterations. Training is also done for one particular optimization/reconstruction algorithm and for each forward operator. This approach lacks an understanding of its recovery guarantees, which is actually a very challenging problem.
	
\paragraph{Generative models}
Instead of using a neural network to learn a penalty functional, it can be used to restrict a prior distribution through a generative model (GAN, variational auto-encoder, diffusion model, etc.) that learns to generate samples similar to some training dataset~\cite{song2021solving,dimakis2022deep,duff2024regularising,chung2022improving,tewari2023diffusion, habring2022generative}. The advantage is that for many problems, the class of vectors to recover lies on a low-dimensional manifold (or rather a  union of finitely many of them) and the training of a generative model amounts to learning a map from a low-dimensional space (the latent space) to the original signal space. However, generative models can suffer from memorization and hallucination, and recovery guarantees have to be properly rethought for such models.
	
\paragraph{Deep Inverse Prior}
In the case where no training data is available, a \textit{self-supervised} method, yet based on a neural network, can be considered. The Deep Inverse Prior (DIP), or Deep Image Prior, was introduced in~\cite{ulyanov_deep_2020} for simple image processing tasks (denoising, super-resolution, and inpainting). The central idea in the DIP is to train a neural network that acts as a generator with a randomly generated input.  It can be thought of as a latent random variable in a lower dimension. The hope is that the architecture of this neural network will induce some ``implicit regularization'' such that more and more detailed content is added during training before eventually overfitting the noisy image. Its theoretical recovery guarantees have been studied extensively in a series of recent papers \cite{buskulic2024convergenceJournal,buskulic2023convergence,buskulic2024descentarxiv,Buskulic25}.

\subsection{Contributions}
Our goal is to recover $\bx$ given $\by$ in \eqref{eq_data_Model} with provable guarantees using the Plug-and-Play (PnP) paradigm \cite{venkatakrishnan_plug-and-play_2013, romano2017little, kamilov2023plug, huraultthesis, TanHNA26}. This approach is inspired by operator splitting schemes, which one often employs to solve variational, model-based problems. In this work, we investigate the forward-backward-splitting (FBS), which in its proximal form is given by the iteration 
	\begin{equation}\label{eq:FBS_general}
		\bx_{k+1} = \prox{\gamma J}{\bx_k -\gamma \nabla_{\bx} F(A\bx_k,\by)},
	\end{equation}
	for some step size $\gamma > 0$, a fidelity term $F$ that is chosen according to the forward model \eqref{eq_data_Model}, and a regularizer $J$, whose proximal mapping is defined by 
	\begin{equation*}
		\prox{\gamma J}{x} \eqdef \Argmin\limits_{z \in \bbR^n} \left\{\frac{1}{2\gamma}\norm{x-z}^2 + J(z)\right\}.
	\end{equation*}
	Choosing $F(A\bx,\by) = \frac{1}{2}\norm{A\bx-\by}^2$ in \eqref{eq:FBS_general} yields \begin{equation}\label{eq:FBS_L2}
		\bx_{k+1} = \prox{\gamma J}{\bx_k -\gamma A^\top(A\bx_k-\by)}.
	\end{equation}
	The idea of PnP is to replace the proximal operator with a suitably chosen operator $\widehat{D}: \R^n \to \R^n$, hence the name `Plug-and-Play'. Using the data model~\eqref{eq_data_Model}, one can then write \eqref{eq:FBS_L2} as \begin{equation}\label{eq:pnp_general}
		\bx_{k+1} = \widehat{D}\pa{\bx + (I_n - \gamma A^\top A)(\bx_k - \bx) + \gamma A^\top\be} .
	\end{equation}
	Therefore, one can intuitively interpret the argument of $\widehat{D}$ as a noisy version of $\bx$, which is the reason why $\widehat{D}$ is often called a "denoiser" in the literature. This intuition is rigorous and exact for the special instance where $A = I_n$, and by setting $\gamma=1$, in which case \eqref{eq:pnp_general} simplifies to
	\begin{equation}\label{eq:pnpden}
		\bx_{k} = \widehat{D}\pa{\bx + \be} , \forall k \in \N .
	\end{equation}
	In particular, the iteration then reduces to a denoising step with noise $\be$.\\
	Note that the iteration \eqref{eq:pnp_general} does not require an explicit regularization functional.  While \eqref{eq:FBS_L2} is tied to the computation of minimizers of the functional $J(\bx) + F(A\bx,\by)$ (see e.g. \cite{BeckFirstOrderOpti}), the PnP iteration \eqref{eq:pnp_general} might not necessarily solve an optimization problem. After all, why should it? In fact, a fixed-point perspective is more appropriate in this setting as we will argue later. Consequently, the argument of the denoiser does not need to be chosen as a gradient-descent step for the fidelity. We therefore consider the more general form \begin{equation}\label{eq:pnpiter}
		\bx_{k+1} = \widehat{D}\pa{\bx_k + \gamma B(\by - A \bx_k)}  \quad \text{ with } \gamma > 0,
	\end{equation} 
	where $B \in \R^{n \times m}$ and its choice gives an additional degree of freedom to the user. For instance, typical choices are $B=A^\top$, $B=A^+$ or $A^\top(A^\top A + \delta I_n)^{-1}$, $\delta > 0$ (Tikhonov regularization). Such choices were used in e.g., \cite{Kim16,Jin17,Zhou22,Wang20,Nguyen26}. Motivated by making the training phase independent of the observation noise covariance matrix $ \Sigma_{\bE}$, \cite{Leterme25} proposed and studied the very interesting choice $B=A^\top \Sigma_{\bE}^{-1/2}$. 
	
	Though general choices of $B$ do not necessarily entail that \eqref{eq:pnp_general} is tied to an optimization problem, this property still holds for some specific choices mentioned above. The case $B=A^\top C$, for $C$ symmetric and positive semidefinite corresponds to gradient descent on $\frac{1}{2}\norm{C^{1/2}(\by - A~\cdot)}^2$. $B=A^+$ corresponds to gradient descent on $\frac{1}{2}\dist(\cdot,\cC)^2$ where $\cC = \enscond{\bz \in \bbR^n}{\by = A\bz}$.
		
	Since the iterates of \eqref{eq:pnpiter} should approximate a solution $\bx$ of \eqref{eq_data_Model}, we eventually want to have $\by - A\bx_k \approx \be$, meaning that $B$ ideally acts as a transformation from the observation noise $\be$ to the type of noise that can be optimally removed by the denoiser $\widehat{D}$.\\
	
	Our analysis focuses on the class of MMSE denoisers of $\bX$ given a random vector $\bZ$. These are defined as \begin{equation}\label{eq:mmse}
		\widehat{D} \in \Argmin_{D \in \scrD(\R^n,\R^n)} \left\{\cR(D) \right\},
	\end{equation} with \begin{equation}\label{eq:mmse_objective}
		\cR(D) \eqdef \Expect{}{\norm{D(\bZ)-\bX}^2},
	\end{equation}
	where the expectation in \eqref{eq:mmse_objective} is taken with respect to the joint distribution of $\bX$ and $\bZ$, and $\scrD(\R^n,\R^n)$ is some subspace of measurable functions\footnote{Borel measurability is not superfluous here; see the counterexample in \cite{Wise85}.}.\\
	
The main contributions of this work is to derive pointwise and distributional recovery error bounds of the iteration \eqref{eq:pnpiter}, where the latter one is measured with respect to the Wasserstein distance. To this end, 
\begin{itemize}
\item We establish several regularity and stability properties of the MMSE denoisers. In particular, we show cocoercivity, which gives a rigorous justification for the firm nonexpansiveness that is enforced in many PnP methods. 
\item We derive error estimates for the PnP iteration with MMSE denoisers from a pointwise perspective and with respect to the Wasserstein distance of the involved probability measures. The latter of those results shows that the transformation $B$ should be chosen such that the transformed noise $B\bE$ is distributed similarly to the noise that is applied to the data when constructing the denoiser, i.e. to $\bZ-\bX$. Moreover,  for optimal performance, the denoiser should be chosen based on the conditioning of $BA$. These results strongly contrast the conventional choice $B = A^\top$, which implicitly assumes that the denoiser in PnP algorithms is agnostic to the forward problem. 
\item We further extend our results to the case where the MMSE denoiser is approximated by a neural network and derive quantitative bounds for the necessary width and depth of the network. Notably, all of our recovery guarantees do not necessarily require non-expansiveness of the denoiser, but rather impose a constraint on the conditioning of the matrix $BA$. 
\end{itemize}


\subsection{Relation to prior work}
Recovery guarantees for PnP remain relatively scarce. While the work~\cite{EbnerHaltmeier24} develops a first regularization theory, it makes strong assumptions on the denoiser which in practical applications may not be satisfied in a computationally feasible way. In particular, the denoiser is assumed to be contractive, and limiting solutions are characterized only for denoisers that essentially behave like proximal mappings. Similarly, in~\cite{hurault_gradient_2021, hurault_nonconvex_2022} the denoiser is designed to be a proximal map for a potentially nonconvex functional.  For denoisers in an $L^\infty$-neighborhood of the identity and forward operators satisfying the restricted eigenvalue conditions, convergence results of the PnP-iteration to a fixed point of the denoiser were established. However, this work does not assess whether these fixed points constitute useful approximations of the sought-for solution. \\
	The PnP algorithm with MMSE denoisers was considered in \cite{laumont2023maximum, xu2020provable} from an optimization perspective. In particular, the authors of \cite{laumont2023maximum} showed that the PnP iteration using approximations of the MMSE denoiser eventually remains in a neighborhood of the maximum a posteriori (MAP) solution for bounded trajectories.\\
	The matrix $B$ in \eqref{eq:pnpiter}, which replaces $A^\top$ in classical forward-backward splitting has been used for a PnP method in \cite{Leterme25}. In this work, the authors choose $B$ such that it has a whitening effect on the noise used for training denoiser. That way, the denoiser becomes independent of the training noise. For MMSE denoisers, it was pointed out in \cite{Nguyen26} that by incorporating such $B$ into the denoiser, one can construct a physics-aware estimator in which the noise is projected onto the image of the forward operator $A$. For certain imaging tasks, such as inpainting, such denoisers are more stable than physics-agnostic ones. Moreover, the use of transformation operators is a popular tool to match the data representation with the implicit prior given by the proximal operator. This interpretation has inspired the authors of \cite{Jin17} to propose a neural network architecture which is based on an unrolling of ISTA.
	
	\subsection*{Notation}
	Throughout this paper, for $d \in \N$, we equip $\R^d$ with the inner product $\dotp{\bx}{\bz}=\bx^\top\bz$ and associated norm $\norm{\bx}=\sqrt{\dotp{\bx}{\bx}}$. The open ball with radius $r$ and center $\bx$ is denoted by $\Ball_r(\bx)$, where we use the shorthand notation $\Ball(\bx)$ if $r = 1$. If $U$ is a positive (semi)definite matrix, we denote the associated (semi)norm by $\norm{x}_U = \sqrt{\dotp{x}{Ux}}$.  For a matrix $A$ we denote the Moore-Penrose pseudo-inverse by $A^+$. If $A$ is quadratic and rank-degenerate, we write its pseudo-determinant (the product of all non-zero eigenvalues of $A$) as $\det(A)_*$. We further denote the largest and smallest eigenvalues of $A$ by $\lambda_{\max}(A)$ and $\lambda_{\min}(A)$, respectively.\\
	For a subset $\cC \subset \R^d$, let $\cB_\cC$ be the Borel $\sigma$-algebra on $\cC$. We define $\cP(\cC)$ as the space of Borel probability measures on $(\cC,\cB_\cC)$ and  $\delta_x$ as the Dirac measure at $x$. The pushforward measure of a measurable function $f$ with respect to a measure $\mu$ is denoted by $f_\#\mu$.  Random vectors are denoted with bold capital letters and their realizations in bold lowercase letters. We further denote the distribution of a random vector $\bX$ by $\mu_\bX$. The support of a Borel measure $\mu$ is denoted by $\supp(\mu)$ and defined as the set of all points for which all open neighborhoods have positive measure.\\
	For $n\in \N$ we set $[n] = \{1,\dots,n\}$. The convex hull of a set $\cC$ is denoted by $\conv{\cC}$. \\
	The Jacobian matrix of a differentiable map $H$ at a point $\bz$ is denoted by $\jac{H}(\bz)$.
	
	\section{The MMSE denoiser}\label{sec:mmsedenoiser}
	We make the assumption 
	\begin{assumption}\label{assum:Xcompact}
		$\cX$ is a compact, 
	\end{assumption}
	\noindent and for the remainder of this work fix $M>0$ with  
	\begin{equation*}
		\max\limits_{\bx \in \cX}\norm{\bx}\le M.
	\end{equation*}  
	We consider \eqref{eq:mmse}, where $\scrD(\R^n,\R^n)$ is the space of measurable functions from $(\cX,\cB_\cX)$ to $(\R^n,\cB_{\R^n})$. In this case, it is well known that the MMSE estimator coincides  $\mu_\bZ$-almost everywhere with the posterior conditional mean (PCM) in Bayesian estimation,
	\begin{equation}\label{eq:dmmse1}
\widehat{D}_{\mu_\bX}(\bz) \eqdef \bbE[\bX | \bZ = \bz ] = \frac{\int_{\cX} \bx d\mu_{\bZ|\bx}(\bz,\bx)d\mu_{\bX}(\bx)}{\int_{\cX} d\mu_{\bZ|\bx}(\bz,\bx)d\mu_{\bX}(\bx)} .
	\end{equation}
	\begin{remark}
If the moments of $\widehat{D}(\bZ)$ in \eqref{eq:mmse} are constrained (typically $\widehat{D}(\bZ)$ is in $L^1(\mu_\bZ)$), then with a density argument of \cite[Lemma~A.1]{Hanneke21}, one could optimize over the space of Lipschitz functions. This is important as in practice, $\widehat{D}$ is computationally intractable (as is the case of the MMSE denoiser \cite{Nguyen26}), and is parametrized with a neural network whose parameters are pre-trained from clean-noisy pairs with a Lipschitz constraint or other structural properties. Nevertheless, we will first stick with the ideal MMSE estimator before moving to NN training that will be the subject of the dedicated Section~\ref{sec:training}.
\end{remark}
	Note that the random variable $\bZ$ can be freely chosen to design an optimal MMSE denoiser.  In view of \eqref{eq:pnpden}, it is natural to consider 
	\begin{equation}\label{eq:noisy}
		\bZ = \bX + \bUpsilon,
	\end{equation}
	where $\bUpsilon$ is a zero-mean random vector independent of $\bX$. Using \eqref{eq_data_Model} to expand \eqref{eq:pnpiter} as in \eqref{eq:pnp_general}, we have \begin{equation}\label{eq:pnpiter2}
		\bx_{k+1} = \widehat{D}_{\mu_\bX}\pa{\bx + (I_n - \gamma B A)(\bx_k - \bx) + \gamma B\be} .
	\end{equation}
	Hence, the distribution of $\Upsilon$ should be equal, or at least close, to that of $\gamma B \bE\sim\cN(0,\gamma^2 B\Sigma_\bE B^\top)$. Since $B$ is not necessarily full column rank, this Gaussian distribution can be degenerate. With this discussion in mind, we will make the following assumption: 
	\begin{assumption}\label{assum:Xnoise}
		$\bUpsilon \sim \cN(0,\Sigma_{\bUpsilon})$, where $\Sigma_{\bUpsilon}$ is symmetric and positive semidefinite. 
	\end{assumption}
	
	\begin{remark}
		This assumption is in stark contrast with most approaches in the literature, where $\Sigma_{\bUpsilon}$ is taken as $\sigma_{\bUpsilon}^2 I_n$, and $\sigma_{\bUpsilon}>0$  is drawn randomly away from $0$ {\cite{venkatakrishnan_plug-and-play_2013,zhang2021plug,hauptmann2025convergent}}. As we will quantify in Section~\ref{sec:recguaranteesMMSE}, the recovery bound depends precisely on how the distribution of the noise $\bUpsilon$ deviates from $\gamma B\bE$. This can be large if $(\gamma B)_{\#}\mu_{\bE}$ and $\mu_{\bUpsilon}$ have very different supports as is the case if $B$ is not full column rank.
	\end{remark}

	Condition \eqref{assum:Xnoise} implies that $\mu_{\bUpsilon}$ is the Gaussian measure that has a density with respect to the Lebesgue measure on the $r$-dimensional subspace $V \eqdef \ran(\Sigma_{\bUpsilon}) \subset \R^n$, with $r=\rank(\Sigma_{\bUpsilon})$. This density reads
	\[
	\phi(\bupsilon;\Sigma_{\bUpsilon}) = \frac{1}{(2\pi)^{r/2}\det(\Sigma_{\bUpsilon})_*^{1/2}} \varphi(\bupsilon;\Sigma_{\bUpsilon}) \qwhereq
	\varphi(\bupsilon;\Sigma_{\bUpsilon}) \eqdef \exp\pa{-\frac{\bupsilon^\top\Sigma_{\bUpsilon}^+\bupsilon}{2}} \mathds{1}_{V}(\bupsilon),
	\]
	In particular this means that $\supp(\phi)\subset\ran(\Sigma_{\bUpsilon})$. \\
	We derive more convenient forms of $\phi$ and $\varphi$. Let $\Pi$ be a $n \times r$ matrix whose columns form an orthonormal basis for $V$. Then the matrix $\widetilde{\Sigma}_{\bUpsilon} = \Pi^\top \Sigma_{\bUpsilon} \Pi$ has full rank $r$ (implying symmetry and positive definiteness), and we get the following equivalent forms of $\phi$ and $\varphi$ 
	\begin{equation}\label{eq:phidef}
		\phi(\bupsilon;\widetilde{\Sigma}_{\bUpsilon}) = \frac{1}{(2\pi)^{r/2}\det(\widetilde{\Sigma}_{\bUpsilon})^{1/2}} \varphi(\bupsilon;\widetilde{\Sigma}_{\bUpsilon}) \qwithq \varphi(\bupsilon;\widetilde{\Sigma}_{\bUpsilon}) = \exp\pa{-\frac{\bupsilon^\top\Pi\widetilde{\Sigma}_{\bUpsilon}^{-1}\Pi^\top\bupsilon}{2}} \mathds{1}_{V}(\bupsilon) . 
	\end{equation}
	
	In this setting, \eqref{eq:dmmse1} reads
	\begin{equation}\label{eq:dmmsegauss}
		\widehat{D}_{\mu_\bX}(\bz) = \frac{\int_{\cX} \bx \varphi(\bz-\bx;\widetilde{\Sigma}_{\bUpsilon}) d\mu_{\bX}(\bx)}{\int_{\cX} \varphi(\bz-\bx;\widetilde{\Sigma}_{\bUpsilon})d\mu_{\bX}(\bx)} . 
	\end{equation}
\begin{remark}
Under suitable conditions, the MMSE denoiser can be linked to the maximum a posteriori (MAP) estimator. 
In particular, it has been shown in \cite{Gribonval2011, Gribnoval13} that when the noise $\bUpsilon$ is Gaussian with non-degenerate covariance, then $\widehat{D}_{\mu_\bX}$ can be written as the proximal mapping of some penalty. In \cite{Pesme25}, the authors show that under strong regularity assumptions on the prior distribution $\mu_\bX$, the iteration $\bx_{k+1} = (1-\alpha_k)\widehat{D}_{\mu_\bX}(\bx_k) + \alpha_k \bz$ approximates the MAP estimator.
\end{remark}

	For $\Sigma_{\bUpsilon} = \sigma^2 I_n$, the MMSE estimator satisfies Tweedie's Identity
	\begin{equation}\label{eq:Tweedie}
		\widehat{D}_{\mu_{\bX}} = I_n + \sigma^2 \nabla \log(p_{\bZ}).
	\end{equation}
	Its scalar version has been derived by Robbins in \cite{Robbins56} with general kernels $\varphi$ beyond the Gaussian case. He credits Tweedie \cite{Tweedie47} who was the first to derive it with a Laplacian-type kernel (see also \cite{Tweedie84} for the exponential family). {Tweedie's formula was also independently discovered by Miyasawa~\cite{miyasawa1961empirical}.} The vector version with $V=\bbR^n$ and the Gaussian kernel was obtained in \cite{Esposito68}. Tweedie's formula was also used in \cite{Guo05,Palomar06,Dytso20} to get derivative formulas in their special non-degenerate cases. As it stands in the intersection of core concepts, including MMSE estimation, score function,  and the optimal Gaussian denoiser, this formula has a strong connection to learning-based methods. In generative diffusion models, it links optimal denoising directions with the objectives of Hyv\"{a}rinen's score matching (SM)~\cite{hyvarinen2005estimation,ghosh2025stein} and denoising score matching (DSM)~\cite{vincent2011connection}. It can also be used to bridge Bayes estimators for naturally corrupted samples and generated images, thus addressing the ambient diffusion problem~\cite{daras2024consistent}. For solving inverse problems via the diffusion-based posterior sampling approach~\cite{chung2022diffusion,boys2023tweedie,nguyen2025training}, Tweedie's posterior mean is used for approximating the likelihood. The application of PnP for Gaussian noise MMSE estimation was explored in~\cite{xu2020provable}, and its link to Tweedie's formula was made explicit in~\cite{pritchard2025nonasymptotic,TanHNA26}. Recently,  Tweedie's formula was extended to infinite dimensional Banach spaces and applied for PDE-based inverse problems~\cite{yao2025guided}. 
	
The estimator \eqref{eq:dmmsegauss} is, in general, difficult to compute, especially in high dimensions. Moreover, in practice, one does not know the distribution $\mu_{\bX}$ but may only be able to sample from it. In other words, one has access to i.i.d. realizations $\cD_N^\bX \eqdef \enscond{\bx^i}{i \in [N]}$ of $\bX$, or, equivalently, to an empirical distribution $\mu^N_\bX = \frac{1}{N}\sum_{i=1}^N \delta_{\bx^i}$. In this case, \eqref{eq:dmmsegauss} becomes
	\begin{equation}\label{eq:Ndmmsegauss}
		\widehat{D}_{\mu_\bX^N}(\bz) = \frac{\int_{\cX} \bx \varphi(\bz-\bx;\widetilde{\Sigma}_{\bUpsilon}) d\mu^N_{\bX}(\bx)}{\int_{\cX} \varphi(\bz-\bx;\widetilde{\Sigma}_{\bUpsilon})d\mu^N_{\bX}(\bx)} = \frac{\sum_{i=1}^N \bx^i \varphi(\bz-\bx^i;\widetilde{\Sigma}_{\bUpsilon})}{\sum_{i=1}^N \varphi(\bz-\bx^i;\widetilde{\Sigma}_{\bUpsilon})} .
	\end{equation}
	Note that \eqref{eq:Ndmmsegauss} is exactly the Nadaraya-Watson estimator with kernel $\phi$. It is also closely related to  the mean-shift, the soft $k$-means, and kernel estimation. Further, it corresponds to the solution of the variational problem \eqref{eq:mmse} when $\mu_\bX$ is replaced by $\mu^N_\bX$. When $\norm{\widetilde{\Sigma}_{\bUpsilon}}$ is small, it can additionally be seen as an approximation to the nearest neighbour denoiser. This assertion will be made rigorous in Section~\ref{subsec:mmselimit}.
	
	\begin{remark}\label{rem:mmseconvX}
		A key observation is that for $\bz \in \supp(\mu_\bX) + V$, $\widehat{D}_{\mu_\bX}(\bz)$ is the barycenter of a probability measure $\mu_{\bX|\bz}$ whose support is contained in $\supp(\mu_\bX)$. It then follows from a standard separation argument that $\widehat{D}_{\mu_\bX}(\bz) \in \conv{\supp(\mu_\bX))}$. Similarly, $\widehat{D}_{\mu_\bX^N}(\bz) \in \conv{\cD_N^\bX} \subset \conv{\supp(\mu_\bX))}$. This implies that the MMSE estimator has a memorization effect.
	\end{remark}

	\section{Main results}
	For notational convenience, we denote the PnP iteration \eqref{eq:pnpiter2} via 
	\begin{equation}\label{eq:pnpiterDNmmse}
		\bx_{k+1} = T_{\mu_{\bX}^N,\gamma}(\bx_k;\by) ,
	\end{equation}
	where
	\[
	T_{\mu_{\bX}^N,\gamma}(\cdot;\by):  \bbR^n \ni \bz  \mapsto \widehat{D}_{\mu_\bX^N} \circ \pa{(I_n - \gamma BA)\bz + \gamma B\by}
	\] 
	for $\gamma > 0$ and $\by \in \bbR^m$. Note that $\widehat{D}_{\mu_\bX^N}$ depends on $\mu_\bX^N$, the empirical measure obtained from the $N$ i.i.d. samples $\cD_N^\bX$ from $\mu_{\bX}$, and thus is random. In turn, so are the iterates even for fixed noise. We will, however, avoid using capital letters in \eqref{eq:pnpiterDNmmse} at this stage.
	
	\subsection{Main assumptions}
	For the remainder of this work, we make the following assumptions.
	\begin{assumption}\label{assum:XV}
		$\bX$ is fully supported within $V$, more precisely it holds $\conv{\supp(\mu_{\bX})} \subset V$, where we recall that $V = \ran(\Sigma_{\bUpsilon})$.
	\end{assumption}

	\begin{assumption}\label{assum:ranB}
		The operator $B$ maps into $V$, i.e. $\ran(B) \subset V$, 
	\end{assumption}
	
	\begin{assumption}\label{assum:BA}
		The matrix $BA$ is symmetric and positive semidefinite. 
	\end{assumption}
	
	\begin{assumption}\label{assum:restinj}
		For $\bx \in \cX$ satisfying \eqref{eq_data_Model}, it holds that
		\[
		\ker(BA) \cap T_{\conv{\supp(\mu_\bX)}}(\bx) = \{0\} ,
		\]
		where $T_{\cC}(\bx) = \cl{\bbR_+\pa{\cC-\bx}}$ is the tangent cone of the convex set $\cC$ at $\bx \in \cC$.
	\end{assumption}

	A discussion of these assumptions is in order. Assumption~\eqref{assum:XV} and Remark~\ref{rem:mmseconvX} imply that
	\begin{equation}\label{eq:Dmmsemu}
		\widehat{D}_{\mu}: \bz \in V \mapsto  V \ni \widehat{D}_{\mu}(\bz) = \frac{\int_{\cX} \bx \varphi(\bz - \bx;\widetilde{\Sigma}_{\bUpsilon}) d\mu(\bx)}{p_\mu(\bz)} \qwithq \varphi(\bupsilon;\widetilde{\Sigma}_{\bUpsilon}) = \exp\pa{-\frac{\bupsilon^\top\Pi\widetilde{\Sigma}_{\bUpsilon}^{-1}\Pi^\top\bupsilon}{2}} ,
	\end{equation}
	where we have used the shorthand definition
	\[
	p_\mu(\bz) \eqdef \int_{\cX} \varphi(\bz - \bx;\widetilde{\Sigma}_{\bUpsilon})d\mu(\bx) .
	\]
	In the context of the MMSE estimator \eqref{eq:dmmsegauss}, assuming \eqref{assum:XV} on $\mu_\bX$ essentially amounts to imposing that the random variable $\bZ = \bX + \bUpsilon$ is contained in the support of the noise distribution, which is a reasonable assumption. In other words, the latent noise $\bUpsilon$ should be chosen such that it can realize any possible ground truth. In particular, this assumption is crucial for establishing Lipschitz continuity of the MMSE denoiser with respect to the prior measure.\\
	Assumption \eqref{assum:ranB} ensures consistency between the (generalized) 'descent' step $\bu_{k} = \bx_k - \gamma B (A\bx_k -\by)$ and the denoising step $\bx_{k+1} = \widehat{D}_{\mu_\bX^N}(\bu_k)$ by forcing the 'descent direction' to be in the range of the signal noise $\bUpsilon$. In view of assumption \eqref{assum:XV}, this ensures that for $\bx_0 \in V$ the argument of the denoiser is well-defined throughout the iteration \eqref{eq:pnpiterDNmmse}. In particular, this assumption is satisfied if $B$ is chosen such that $\ran(B) = V$, which in turn means $\ran(B) = \ran(\Sigma_{\bUpsilon})$.\\
	Assumption~\eqref{assum:BA} entails that the forward operator part in $T_{\mu_{\bX}^N,\gamma}(\cdot;\by)$ is a gradient descent corresponding to the convex function $\bz \in \bbR^n \mapsto \frac{1}{2}\dotp{BA\bz}{\bz} - \dotp{B\by}{\bz}$, and assumption~\eqref{assum:restinj} ensures invertibility of $BA$ on $T_{\conv{\supp(\mu_\bX)}}(\bx)$. Moreover assumptions \eqref{assum:BA}-\eqref{assum:restinj} are used to ensure contractiveness of the operator $T_{\mu_{\bX}^N,\gamma}(\cdot;\by)$.  
	
	
	Consider the case $B=A^\top C$ or $B=A^+$, where $C$ is a symmetric positive definite matrix, e.g., Tikhonov regularization or $C=\Sigma_{\bE}^{-1/2}$ as in \cite{Leterme25}. If the inclusion in \eqref{assum:ranB} becomes an equality, then \eqref{assum:XV} implies that \eqref{assum:restinj} is superfluous. Indeed, in this case $V=\ran(B)=\ran(A^\top)=\ker(BA)^\perp$, and thus $\pa{\bx_k - \bx}_{k \in \bbN} \subset \ker(BA)^\perp$ thanks to \eqref{assum:XV}. We recover the setting considered in \cite{Leterme25}. Additionally, if the inclusion in \eqref{assum:ranB} is strict, then \eqref{assum:restinj} is needed.

	\subsection{Recovery guarantees for PnP with the MMSE}\label{sec:recguaranteesMMSE}
	In this section, we present recovery bounds for the iteration \eqref{eq:pnpiter2}, for which we provide self-contained proofs in the dedicated Section~\ref{sec:proofs}. The starting point of our analysis is the following existence result.
	\begin{lemma}\label{lem:fixpnpfbs}
		Assume that \eqref{assum:Xcompact}--\eqref{assum:Xnoise} and \eqref{assum:ranB} hold, and that  $\mu_{\bX}$ fulfills \eqref{assum:XV}. Then the set of fixed points of $T_{\mu_{\bX}^N,\gamma}$ is non-empty and compact.
	\end{lemma}
	
	\begin{proof}
		Under our assumptions, $T_{\mu_{\bX}^N,\gamma}(\cdot;\by)$ is Lipschitz continuous thanks to Corollary~\ref{lem:Dmulip}. Moreover, from \eqref{assum:XV}, \eqref{assum:ranB} and Remark~\ref{rem:mmseconvX}, it maps $\conv{\supp(\mu_{\bX})}$ onto $\conv{\supp(\mu_{\bX})}$. Since $\supp(\mu_\bX)$ is compact (it is closed by definition and hence compact by \eqref{assum:Xcompact}), it follows from \cite[Theorem~3.20]{Rudin91} that $\conv{\supp(\mu_{\bX})}$ is a compact convex set. The claim is then ensured by Brouwer's fixed point theorem. 
	\end{proof}
	
	\subsubsection{Pointwise recovery error bound }
	
	As can be seen in Lemma~\ref{lem:Dmulip} below, the empirical MMSE denoiser is Lipschitz continuous. We denote its Lipschitz constant by $\lip{\widehat{D}_{\mu_{\bX}^N}}$ and derive the following error bound for the iteration \eqref{eq:pnpiterDNmmse}:
	\begin{align}
		\norm{\bx_{k+1} - \bx}
		&= \norm{\widehat{D}_{\mu_\bX^N}\pa{\bx_k + \gamma B(\by - A \bx_k)} - \bx} \nonumber\\
		&= \norm{\pa{\widehat{D}_{\mu_\bX^N}\pa{\bx + \pa{I_n - \gamma B A}(\bx_k - \bx) + \gamma B\be} - \widehat{D}_{\mu_\bX^N}\pa{\bx}} + \pa{\widehat{D}_{\mu_\bX^N}\pa{\bx} - \bx}} \nonumber\\
		&\leq \lip{\widehat{D}_{\mu_{\bX}^N}} \norm{(I_n - \gamma B A)(\bx_k - \bx) + \gamma B\be} + \norm{\widehat{D}_{\mu_\bX^N}\pa{\bx} - \bx} \nonumber\\
		&\leq \lip{\widehat{D}_{\mu_{\bX}^N}} \norm{(I_n - \gamma B A)(\bx_k - \bx)} + \gamma\lip{\widehat{D}_{\mu_{\bX}^N}} \norm{B}\norm{\be} + \norm{\widehat{D}_{\mu_\bX^N}\pa{\bx} - \bx} . \label{eq:bndpnpfbs_1}
	\end{align}
	Note that \eqref{eq:bndpnpfbs_1} consists of three terms: the first one pertains to the iteration error, the second one originates from the noise in the observations, and the third one captures the performance of the denoiser. Theorem~\ref{thm:pointwiserecovbndpnpfbs} below bounds these terms individually to obtain a recovery bound for the PnP iteration. 
	
Before stating it, we recall that the minimal conical singular value of a matrix $M$ with respect to a closed cone $\cK$ is defined as
\[
\lambda_{\min}(M;\cK) = \inf_{\bz \in \cK \setminus \{0\}} \norm{M \bz}/\norm{\bz} = \min_{\bz \in \cK \cap \sph^{n-1}} \norm{M \bz} .
\]
The last equality follows from positive homogeneity, and is indeed a minimum as the norm is continuous and the constraint set is compact. The condition number of $BA$ on the cone $T_{\conv{\supp(\mu_\bX)}}(\bx)$ is 
\begin{equation}\label{eq:condnum}
		\kappa(BA) \eqdef \frac{\lambda_{\max}(BA)}{\lambda_{\min}\pa{BA;T_{\conv{\supp(\mu_\bX)}}(\bx)}} \geq 1 ,
	\end{equation}
	Note that $\kappa(BA)$ is well-defined due to assumption \eqref{assum:restinj}.

	\begin{theorem}\label{thm:pointwiserecovbndpnpfbs}
		Assume that \eqref{assum:Xcompact}--\eqref{assum:Xnoise} and \eqref{assum:ranB}--\eqref{assum:restinj} hold and that $\mu_{\bX}$ fulfills \eqref{assum:XV}. Suppose also that either 
		\begin{equation}\label{eq:cond_nonexpansive}
			\frac{M^2}{\lambda_{\min}(\widetilde{\Sigma}_{\bUpsilon})} \leq 1,
		\end{equation} or otherwise, that 
		\begin{equation}\label{eq:cond_welldefined}
			\kappa(BA) < \pa{1 - \frac{\lambda_{\min}(\widetilde{\Sigma}_{\bUpsilon})^2}{M^4}}^{-1/2} 
		\end{equation}
		holds. Then, there exists $\epsilon > 0$ such that the PnP-FBS iterates with \begin{equation}\label{eq:optgamma}
			\gamma = \pa{\kappa(BA)\lambda_{\max}(BA)}^{-1}
		\end{equation} obey
		\begin{multline}\label{eq:pointwise_recovery_1}
			\norm{\bx_k - \bx}
			\leq q^{k+1}\norm{\bx_0 - \bx} + \frac{1-q^{k}}{1-q}\Bigg(\frac{M^2}{\lambda_{\min}(\widetilde{\Sigma}_{\bUpsilon})}\gamma\norm{B}\norm{\be} +  2M\pa{N-1} e^{-\frac{\epsilon}{2\lambda_{\max}(\widetilde{\Sigma}_{\bUpsilon})}} \\ + \sqrt{\frac{\lambda_{\max}(\widetilde{\Sigma}_{\bUpsilon})}{\lambda_{\min}(\widetilde{\Sigma}_{\bUpsilon})}}\dist(\bx,\cD_N^\bX)\Bigg) ,
		\end{multline}
		where $q = \frac{M^2}{\lambda_{\min}(\widetilde{\Sigma}_{\bUpsilon})} \sqrt{1 - \kappa\pa{BA}^{-2}}\in [0,1[$.\\
		Moreover, if $\mu_\bX$ has a density with respect to the volume measure on $V$ which is bounded away from zero, then for $N$ large enough, it holds with probability at least $1 - K_2 N^{-2}$ over the sampling of $\cD_N^\bX$ that 
		\begin{multline}\label{eq:pointwise_recovery_2}
			\norm{\bx_k - \bx}
			\lesssim q^{k+1}\norm{\bx_0 - \bx} + \frac{1-q^{k}}{1-q}\Bigg(\frac{M^2}{\lambda_{\min}(\widetilde{\Sigma}_{\bUpsilon})}\gamma\norm{B}\norm{\be} +  2M\pa{N-1} e^{-\frac{\epsilon}{2\lambda_{\max}(\widetilde{\Sigma}_{\bUpsilon})}} \\ + K_1 3^{1/r} \sqrt{\frac{\lambda_{\max}(\widetilde{\Sigma}_{\bUpsilon})}{\lambda_{\min}(\widetilde{\Sigma}_{\bUpsilon})}}\pa{\frac{\log(N)}{N}}^{1/r} \Bigg) ,
		\end{multline}
		where $r = \dim{V}$, and $K_1, K_2 > 0$ are constants that depend only on $r$ and $M$.
	\end{theorem}
	
	Estimate \eqref{eq:pointwise_recovery_1} implies \begin{equation*}
		(1-q)\limsup_{k \to +\infty} \norm{\bx_k - \bx} \leq \frac{M^2}{\lambda_{\min}(\widetilde{\Sigma}_{\bUpsilon})}\gamma\norm{B}\norm{\be} +  2M\pa{N-1} e^{-\frac{\epsilon}{2\lambda_{\max}(\widetilde{\Sigma}_{\bUpsilon})}} + K_1 3^{1/r}\sqrt{\frac{\lambda_{\max}(\widetilde{\Sigma}_{\bUpsilon})}{\lambda_{\min}(\widetilde{\Sigma}_{\bUpsilon})}}\dist(\bx,\cD_N^\bX) .
	\end{equation*}
	The right-hand side consists of three terms. The first one can be interpreted as a noise amplification and depends linearly on the norm of the noise, as it is often the case in regularization theory \cite{engl1996regularization, scherzer2009variational, RegularizationInBanach}. The second and third term capture both the MMSE denoiser performance and sampling error. The last term reflects the fact that as the noise decreases, the PnP-FBS behaves similarly to a nearest neighbor denoiser, hence provably showing the memorization effect that limits PnP.

	\begin{remark} \label{rem:assumptions_convergence}
		\begin{enumerate}[label = (\roman*)]
			\item  As opposed to the classical results for forward-backward splitting and many analyses of its PnP version (see e.g. \cite{Leterme25, EbnerHaltmeier24, hurault_gradient_2021, hurault_nonconvex_2022}) we do not assume that the denoiser is contractive or firmly non-expansive.  If $\widehat{D}_{\mu_\bX^N}$ is not contractive, Theorem~\ref{thm:pointwiserecovbndpnpfbs} still gives a recovery bound, provided that \eqref{eq:cond_nonexpansive} holds. This condition imposes that $BA$ cannot be too ill-conditioned. Note, however, that the left-hand side in \eqref{eq:cond_nonexpansive} refers to the condition number of $BA$ on the cone $T_{\conv{\supp(\mu_\bX)}}(\bx)$, which may differ significantly from the classical condition number. In other words, $BA$ only needs to be well-conditioned (up to some shifting) on the set of reasonable solutions. Moreover, interpreting the Lipschitz bound of the denoiser as an upper bound on the SNR, condition \eqref{eq:cond_nonexpansive} relates this to the conditioning of the operator $BA$: the higher the SNR, the better the conditioning should be. Equivalently, for fixed conditioning, the denoiser is allowed to be only mildly expansive. As discussed after Lemma~\ref{lem:Dmulip} below, the Lipschitz constant bound degrades as the noise decreases, which, in turn, means less stability of the denoiser. Thus, the noise should not be amplified too much by $B$ to maintain overall stability of the PnP-FBS iteration.  
			\item In \cite{liu2021recovery} the authors also derive recovery bounds for the PnP iteration with $B = A^\top$ and for potentially non-contractive denoisers. In particular, they derive a stepsize ensuring convergence of the iterates under the condition \begin{equation*}
				\kappa(A^\top A) \le 1+\frac{2}{\lip{D}},
			\end{equation*} 
			where $\lip{D}$ is the Lipschitz constant of the denoiser. Compared to \eqref{eq:cond_welldefined}, their estimate allows for a slightly worse conditioning of the forward problem if the denoiser is fixed. However, in our setting, this limitation may be alleviated by choosing the matrix $B$ appropriately, thus yielding a preconditioning of the problem. Further, the results in \cite{liu2021recovery} assume a uniform bound of the residual operator $I-D$. This is true in the setting of this work due to \eqref{assum:Xcompact}, but in general, it may be hard to verify. Similarly, in \cite{Leterme25} the authors consider the choice $B = A^\top \pa{\widetilde{\Sigma}_{\bE}}^\frac{1}{2}$ and derive a stepsize based on the conditioning of the operator $BA$. However, the analysis of the PnP iteration is carried out for nonexpansive denoisers.
		\item The quantity $\epsilon$ in Theorem~\ref{thm:pointwiserecovbndpnpfbs} depends in particular on the sample $\cD_N^\bX$, hence on $N$. This dependence can be removed; see Remark~\ref{rem:epsN}.
		\end{enumerate}
	\end{remark}
	
	\subsubsection{Wasserstein recovery error bound}\label{subs:Wasserstein}
	Formally, we would expect $\widehat{D}_{\mu_{\bX}^N}$ to converge pointwise almost surely to $\widehat{D}_{\mu_{\bX}}$ as $N\to \infty$. Furthermore, for sufficiently well-behaved distributions $\mu_{\bX}$, the MMSE denoiser becomes an approximation of identity, as can be seen by \eqref{eq:dmmse1}. Thus, one has convergence to the identity almost everywhere. If this behavior was quantifiable, it could simplify the estimate \eqref{eq:bndpnpfbs_1} and the subsequent analysis of the PnP iteration. However, this kind of convergence is not topologizable, and we therefore cannot obtain quantitative convergence results. This can also be seen in the bound \eqref{eq:pointwise_recovery_2}, which blows up as $N\to\infty$. Thus, in order to improve the result, we change the perspective and focus on distributional error bounds rather than pointwise ones. This means instead of pointwise bounds for the iterates of the PnP given a realization $\by$ of $\bY$, we now turn to establishing bounds on the expected Wasserstein distance between the distribution $\widehat{\mu}_{k}^N$ of the PnP iterates and the empirical MMSE denoiser and the distribution $\mu_{\bX}$ of $\bX$.\\
	Recall that for $p \in [1,\infty[$, the $p$-Wasserstein distance is a metric on $\cP(\cX)$, defined for $\mu,\nu \in \cP(\cX)$ as
	\[
	\cW_p(\mu,\nu) \eqdef \inf\{\bbE[\norm{\bU-\bV}^p]^{1/p}: \bU \sim \mu, \bV \sim \nu\} ,
	\]
	where the infimum is over all joint distributions of $(\bU,\bV)$ (i.e., couplings of $\mu$ and $\nu$). It is known that $\cW_p$ is a metric on probability measures $\cP(\cX)$ with finite $p$-th moment and that it metrizes weak convergence \cite{VillaniON09}. It follows immediately from H\"older's inequality that
	\begin{equation}\label{eq:Wporder}
		\cW_p \leq \cW_q \qforq p \leq q .
	\end{equation}
	As for the pointwise recovery bounds, the starting points of our analysis will be an error bound similar to \eqref{eq:bndpnpfbs_1}. However, since we aim to state a recovery result with respect to Wasserstein distances, we consider iteration \eqref{eq:pnpiter2} with random vectors. Thus, for a fixed realization of the sampling $\cD_N^\bX$, let $\bX \sim \mu_{\bX}, \bUpsilon\sim\cN(0,\Sigma_{\bUpsilon})$, $\bE \sim \mu_{\bE}$ and $\bX^N \sim \mu_\bX^N$. Further denote by $\bX_k$ the iterates of \eqref{eq:pnpiterDNmmse} with $\bY = A\bX + \bE$ instead of $\by$. It holds 
	\begin{align}
		\norm{\bX_{k+1} - \bX}
		&= \norm{\widehat{D}_{\mu_{\bx}^N}\pa{\bX_k+\gamma B(\bY-A\bX_k)}-\bX}\nonumber\\
		&\leq  \norm{\widehat{D}_{\mu_{\bx}^N}\pa{\bX+(I_n - \gamma B A)(\bX_k - \bX) + \gamma B\bE} - \widehat{D}_{\mu_{\bx}^N}\pa{\bX+\bUpsilon}} \nonumber \\
		&+ \norm{\widehat{D}_{\mu_{\bx}^N}(\bX+\bUpsilon) - \widehat{D}_{\mu_{\bx}^N}(\bX^N+\bUpsilon)} + \norm{\widehat{D}_{\mu_\bX^N}\pa{\bX^N + \bUpsilon} - \bX^N} + \norm{\bX^N - \bX} \nonumber\\
		&\leq \lip{\widehat{D}_{\mu_{\bX}^N}}\norm{(I_n - \gamma B A)(\bX_k - \bX)} + \lip{\widehat{D}_{\mu_{\bX}^N}}\norm{\gamma B\bE - \bUpsilon} + \norm{\widehat{D}_{\mu_\bX^N}\pa{\bX^N + \bUpsilon} - \bX^N} \nonumber\\
		&\quad + \pa{1+\lip{\widehat{D}_{\mu_{\bX}^N}}}\norm{\bX^N - \bX} . \label{eq:bndpnpfbs_2}
	\end{align}
	This error bound entails four terms. As in \eqref{eq:bndpnpfbs_1}, the first one pertains the iteration error. The second term describes the difference between $\gamma B\bE$ and $\bUpsilon$, see also the discussion in Section~\ref{sec:mmsedenoiser}. If one chooses $B = A^\top$, i.e. the argument in \eqref{eq:pnpiter2} is of gradient descent type, it establishes a direct link between the forward operator $A$ and $\bUpsilon$ (which determines the denoiser). This observation is crucial since it shows that the PnP framework is not agnostic to the ill-posed problem. Rather, the denoiser needs to be chosen such that it is compatible with the forward operator. The third term in \eqref{eq:bndpnpfbs_2} captures the performance of the denoiser, and the fourth term corresponds to the sampling error when constructing the denoiser.\\
	We further need to uniformize \eqref{assum:restinj}. For this, let
	\[
	\widetilde{\lambda}_{\min}(BA) \eqdef \inf_{\bz \in {\supp(\mu_\bX)}} \lambda_{\min}\pa{BA;T_{\conv{\supp(\mu_\bX)}}(\bz)} .
	\] 
	To ensure $\widetilde{\lambda}_{\min}(BA)>0$, we need a stronger assumption than \eqref{assum:restinj}.
	\begin{altassumption}{assum:restinj}\label{assum:restinj_unif}
		\begin{equation*}
			\supp(\mu_{\bX}) \cap \ri(\conv{\supp(\mu_{\bX})}) \neq \emptyset
		\end{equation*}
		and 
		\begin{equation*}
			\ker(BA) \cap \Par(\conv{\supp(\mu_{\bX})}) = \{0\} ,
		\end{equation*}
		where for a non-empty convex set $\cC$, $\Par(\cC)$ is the subspace parallel to $\cC$, i.e., for any point $\bx \in \cC$, we have
			\[
			\Par(\cC) = \Aff(\cC) - \ens{\bx} = \bbR(\cC - \cC) = \bbR_+(\cC - \cC) .
			\]
	\end{altassumption}

	
	\begin{lemma}\label{lem:uniformization_restricted_injectivity}
		Under assumptions \eqref{assum:Xcompact} and \eqref{assum:restinj_unif} it holds $\widetilde{\lambda}_{\min}(BA) >0$.
	\end{lemma}

	\begin{remark}
		Assumption \eqref{assum:restinj_unif} can be weakened if one assumes that the restricted injectivity property holds for all $\bx \in \supp(\mu_{\bX})$ and that the core outer limit (see \cite{Luc_2008}) of the tangent cones is contained in their union. A detailed statement and proof can be found in appendix~\ref{app:general_uniformization}.
	\end{remark}
	Define $\widetilde{\kappa}(BA)$ and $\widetilde{\gamma}$ accordingly by replacing $\lambda_{\min}\pa{BA;T_{\conv{\supp(\mu_\bX)}}(\bz)}$ by $\widetilde{\lambda}_{\min}(BA)$ in \eqref{eq:condnum}, and \eqref{eq:optgamma} respectively. Lastly, we recall that the Minkowski dimension of  a measure $\mu$ is given as \[
	d_{\mathrm{M}}(\mu) \eqdef \inf\{\dim_\mathrm{M}(\cC): \mu(\cC) = 1\},
	\]
	where $\dim_\mathrm{M}(\cC)$ is the Minkowski-Bouligand or box-counting dimension \cite{Falconer03}. The latter is defined as  
	\begin{equation}
		\dim_\mathrm{M}(\cC) := \lim\sup_{\epsilon\to 0^+}\frac{\log N_{\epsilon}(\cC)}{-\log\epsilon},
	\end{equation}
	and $N_{\epsilon}(\cC)$ is the $\epsilon$-covering number of the set $\cC$ given by the smallest $m$ such that there exist $m$ closed balls $\cl{\Ball}_1,\dots,\cl{\Ball}_m$  of diameter $\epsilon$ covering $\cC$, i.e.,   $\cC\subseteq \bigcup_{i=1}^m \cl{\Ball}_i$.\\
	We are now ready to state our main recovery theorem. For this, we emphasize the dependence of the iterates $\bX_k$ on the random variables from which it is built, i.e., 
	\begin{equation}\label{eq:xkx0y}
	\bX_k(\bX_0,\bY) = T_{\mu_{\bX}^N,\gamma}^k(\bX_0;\bY)=T_{\mu_{\bX}^N,\gamma}^k(\bX_0;A\bX+\bE) ,
	\end{equation}
	where $\bX \sim \mu_{\bX}$ and $\bE$ satisfies \eqref{assum:noiseonY}, and we suppose without loss of generality that $\bX_0$ is drawn randomly from $\mu_\bX^N$ but is independent from all other variables. The iterate $\bX_k$ also depends on the denoiser $\widehat{D}_{\mu_{\bX}^N}$, hence on the empirical measure $\mu_{\bX}^N$ constructed from the (random) samples $\cD_N^\bX$. We let $\widehat{\mu}_k^N$ be the law of $\bX_k$ conditionally on $\cD_N^\bX$. \\

	\begin{theorem}\label{thm:Wrecovbndpnpfbs}
		Assume that \eqref{assum:noiseonY}--\eqref{assum:Xnoise},\eqref{assum:ranB}--\eqref{assum:BA}, and \eqref{assum:restinj_unif} hold, and that $\mu_{\bX}$ fulfills \eqref{assum:XV}. Suppose also that either \begin{equation}
			\frac{M^2}{\lambda_{\min}(\widetilde{\Sigma}_{\bUpsilon})} \leq 1,
		\end{equation}
		or otherwise, that 
		\[
		\widetilde{\kappa}(BA) < \pa{1 - \frac{\lambda_{\min}(\widetilde{\Sigma}_{\bUpsilon})^2}{M^4}}^{-1/2}.
		\]
		Additionally assume $d_{\mathrm{M}}(\mu_\bX) \geq 4$ and let $s> d_{\mathrm{M}}(\mu_\bX)$. Then it holds for the PnP-iteration with stepsize $\widetilde{\gamma} = \pa{\widetilde{\kappa}(BA)\lambda_{\max}(BA)}^{-1}$ that
		\begin{align}\label{eq:Wasserstein_recovery}
			\Expect{\cD_N^\bX}{\cW_2(\widehat{\mu}_k^N,\mu_{\bX})}
			&\lesssim \widetilde{q}^{k+1}\sqrt{\Expect{}{\norm{\bX_0 - \bX}^2}} \notag\\
			+ \frac{1-\widetilde{q}^{k}}{1-\widetilde{q}}\Bigg[&
			\frac{M^2}{\lambda_{\min}(\widetilde{\Sigma}_{\bUpsilon})}\pa{\widetilde{\gamma}^{2}\tr{B \Sigma_\bE B^\top} + \tr{\widetilde{\Sigma}_{\bUpsilon}} 
				- 2\widetilde{\gamma}\tr{\widetilde{\Sigma}_{\bUpsilon}^{1/2}\Pi^\top B \Sigma_\bE B^\top\Pi\widetilde{\Sigma}_{\bUpsilon}^{1/2}}^{1/2}}^{1/2} \notag\\
			&+ \sqrt{\lambda_{\max}(\widetilde{\Sigma}_{\bUpsilon}) r}
			+ \pa{1+\frac{M^2}{\lambda_{\min}(\widetilde{\Sigma}_{\bUpsilon})}}\diam(\cX)N^{-1/s} \Bigg],
		\end{align}
		where $\widetilde{q} = \frac{M^2}{\lambda_{\min}(\widetilde{\Sigma}_{\bUpsilon})} \sqrt{1 - \widetilde{\kappa}\pa{BA}^{-2}} \in [0,1[$, and $\Expect{\cD_N^\bX}{\cdot}$ denotes the expectation with respect to the sampling.
	\end{theorem}
	
	\begin{remark}
		Estimate \eqref{eq:Wasserstein_recovery} implies
		\begin{multline*}
			\limsup_{k \to +\infty, N \to +\infty} (1-\widetilde{q})\Expect{\cD_N^\bX}{\cW_2(\widehat{\mu}_k^N,\mu_{\bX})} 
			\lesssim \frac{M^2}{\lambda_{\min}(\widetilde{\Sigma}_{\bUpsilon})}\Bigg(\widetilde{\gamma}^{2}\tr{B \Sigma_\bE B^\top} + \tr{\widetilde{\Sigma}_{\bUpsilon}} \\
			- 2\widetilde{\gamma}\tr{\widetilde{\Sigma}_{\bUpsilon}^{1/2}\Pi^\top B \Sigma_\bE B^\top\Pi\widetilde{\Sigma}_{\bUpsilon}^{1/2}}^{1/2}\Bigg)^{1/2}
			+ \sqrt{\lambda_{\max}(\widetilde{\Sigma}_{\bUpsilon}) r} .
		\end{multline*}
		The term encoding the discrepancy between the two Gaussian noises $\widetilde{\gamma} B\bE$ and $\bUpsilon$ is known as the Bures metric, see e.g., \cite{Gelbrich90}. Notably, it depends on the step size $\widetilde{\gamma}$. This suggests that the covariance $\widetilde{\Sigma}_{\bUpsilon}$ should depend on $\widetilde{\gamma}$ as well to make this term small or even vanish. Intuitively, this reflects the fact that the denoiser has to account somehow for the step-size in a similar vein as for the proximal mapping with which the analogy is usually made, though this analogy is questionable as we argued before. As our analysis shows (compare also to the definition of $\widetilde{\gamma}$), the optimal step size depends on the conditioning of the inverse problem. This means that $\widetilde{\Sigma}_{\bUpsilon}$, hence the denoiser, should also depend on this conditioning. Again, this stands in contrast with the widespread idea that the denoiser used in PnP methods can be designed independently of the underlying inverse problem. In particular, the choice $\Sigma_{\bUpsilon} = \widetilde{\gamma}^{2}B \Sigma_\bE B^\top$ yields 
		\begin{equation}
			\limsup_{k \to +\infty, N \to +\infty} \Expect{\cD_N^\bX}{\cW_2(\widehat{\mu}_k^N,\mu_{\bX})} \lesssim (1-\widetilde{q})^{-1}\widetilde{\gamma} \norm{B} \sqrt{\lambda_{\max}(\Sigma_\bE) r} ,
		\end{equation}
		where now $r = \rank(B)$. Moreover, with the choice $\Sigma_{\bUpsilon} = \widetilde{\gamma}^{2}B \Sigma_\bE B^\top$ for $B = A^\top C$ or $B=A^+C$ and $C = \Sigma_{\bE}^{-\frac{1}{2}}$ or $C = (I + A^\top A^{-1})$ (see \cite{Leterme25}), assumption \eqref{assum:restinj} is not needed as mentioned above. In this case, one could also replace the conical minimal singular value $\widetilde{\lambda}(BA)$ by the smallest singular value of $BA$ over its whole domain. This, however yields a smaller step size $\widetilde{\gamma}$.
		On the other hand, if one chooses $\Sigma_{\bUpsilon} = \sigma^2I_r$ for some $\sigma>0$, the noise discrepancy in the bound \eqref{eq:Wasserstein_recovery} will not vanish. In fact, as in the pointwise analysis, one cannot afford adjusting freely the variance of the noise $\bUpsilon$ when designing the MMSE denoiser to make it arbitrarily contractive as this will harm the the error term in the Bures metric between $\widetilde{\gamma}B\bE$ and $\bUpsilon$.
	\end{remark}

\begin{remark}
The scaling in $N$ in \eqref{eq:Wasserstein_recovery} can be sharpened with a finer notion of dimension. See the discussion in Remark~\ref{rem:Wdiscrete}.
\end{remark}

In addition to proving the bound on the expected value of the quantity $\cW_2(\widehat{\mu}_k^N,\mu_{\bX})$, we also show a high-probability error bound.
	
	\begin{proposition}\label{prop:high_prob_Wasserstein}
	Suppose that the assumptions of Theorem~\ref{thm:Wrecovbndpnpfbs} hold and that $d_{\mathrm{M}}(\mu_\bX) \geq 2$. Let $s> 2d_{\mathrm{M}}(\mu_\bX)$. It holds with probability $1-\exp(-2N^{1-4/s})$ over the sampling of $\cD_N^\bX$, that
		\begin{align*}
			\cW_2(\widehat{\mu}_k^N,\mu_{\bX}) &\lesssim \Bigg( \widetilde{q}^{k+1}\sqrt{\Expect{}{\norm{\bX_0 - \bX}^2}} \notag\\
			&+ \frac{1-\widetilde{q}^{k}}{1-\widetilde{q}}\Bigg[
			\frac{M^2}{\lambda_{\min}(\widetilde{\Sigma}_{\bUpsilon})}\pa{\widetilde{\gamma}^{2}\tr{B \Sigma_\bE B^\top} + \tr{\widetilde{\Sigma}_{\bUpsilon}} 
				- 2\widetilde{\gamma}\tr{\widetilde{\Sigma}_{\bUpsilon}^{1/2}\Pi^\top B \Sigma_\bE B^\top\Pi\widetilde{\Sigma}_{\bUpsilon}^{1/2}}^{1/2}}^{1/2} \notag\\
			&+ \sqrt{\lambda_{\max}(\widetilde{\Sigma}_{\bUpsilon}) r}
			+ \pa{1+\frac{M^2}{\lambda_{\min}(\widetilde{\Sigma}_{\bUpsilon})}}\diam(\cX)N^{-1/s} \Bigg] \Bigg) .
		\end{align*} 
	\end{proposition}
	
	\subsection{Recovery guarantees for PnP with neural network denoiser}\label{sec:training}
	While the empirical MMSE denoiser enjoys the recovery guarantees from the previous section, its evaluation is computationally not feasible for large sample sizes and in high dimension. We therefore extend our analysis where the MMSE is replaced by a denoiser parametrized by a trained neural network. In view of assumption \eqref{assum:XV}, we identify $V$ with $\R^r$ and consider the class of ReLU neural networks $D: V \to V$  with $L = \mathrm{depth}(D)\in \N$ layers. These are recursively defined via
	\begin{align*}
		&D_0(\bx) = \bx\\
		&D_l(\bx) = \max(0,A_lD_{l-1}(\bx) + \bb_l)\quad \text{ for } l = 1,\dots,{L-1}\\
		&D(\bx) = A_L D_{L-1}(\bx) + \bb_L,
	\end{align*}
	where $A_l \in \R^{n_l\times n_{l-1}}$ with $n_0 =n_L = r$, and $\bb_l \in \R^{n_l}$ and the maximum is taken component-wise. We denote the weights of $D$, i.e. the total number of non-zero entries in $A_1,\dots,A_L, \bb_1,\dots,\bb_L$, by $W(D)$. 
	
	As for the empirical MMSE denoiser, we assume that we have access to an i.i.d. sample $\cD_N^{\bX} =\enscond{\bx^i}{i \in [N]} $ from $\mu_\bX$. We further sample i.i.d. noise vectors $\cD_N^{\bUpsilon} = \enscond{\bupsilon^i}{i \in [N]}$ from $\cN(0,\Sigma_{\bUpsilon})$. Following \eqref{eq:mmse}, we now parametrize the class $\scrD(\R^n,\R^n)$ by neural networks. We therefore consider the denoiser $\widehat{D}^N$ defined as 
	\begin{equation}\label{eq:loss}
		\widehat{D}^N \in \Argmin_{D \in \cN\cN(W,L,\beta)}\sum\limits_{i = 1}^N \norm{D(\bx^i+\bupsilon^i) - \bx^i}^2 ,
	\end{equation}
	where the class of feasible denoisers is
	\begin{align}\label{eq:NN_feasible_set}
		\cN\cN(W,L,\beta) \eqdef \big\{&D:V\to V  \text{ is a ReLU neural network with} \notag\\&W(D)\le W,\, \mathrm{depth}(D) \le L, \text{ whose weights are bounded by }\beta \text{, and }  D(\cQ)\subset \cQ \big\} .
	\end{align}
	Here $\cQ$ is a closed cube with $ {{\supp(\mu_{\bX})}}\subset\cQ$. In particular, we may assume that the sidelengths of $\cQ$ are bounded by $\diam(\cX) \leq 2M$. Note that any $D \in \cN\cN(W,L,\beta)$ is Lipschitz continuous, but the exact computation of its Lipschitz constant is NP-hard \cite{Scaman2018}. \\
	
	In order to make use of quantitative approximation estimates for Lipschitz functions on manifolds with neural networks, we make the following assumption.
	\begin{assumption}\label{assum:manifold}
		$\cS\eqdef {{\supp(\mu_{\bX})}}$ is a $d_{\cS}$-dimensional manifold.
	\end{assumption}
	Note that by compactness of $\supp(\mu_{\bX})$, for any $\delta >0$, the manifold $\cS$ is $(R,\delta)$-covered, i.e. there exist $R$ open balls $\Ball_1,\dots,\Ball_R$ of radius $\frac{\delta}{4}$, centered in $\cS$ with \begin{equation*}
		\cS \subseteq \bigcup\limits_{i = 1}^R \Ball_i.
	\end{equation*} 
	In view of \eqref{eq:loss}, a good denoiser should be a stable approximation of the identity on $\cS$. This is the rationale behind the following approximation lemma which motivates the choice of hyperparameters in the class $\cN\cN(W,L,\beta)$ in \eqref{eq:NN_feasible_set}.
	
	\begin{lemma}\label{lem:NN_existence_uniform}
		Assume that \eqref{assum:manifold} holds. Let $\delta>0$ and $R\in \N$ such that $\cS$ is $(R,\delta)$-covered. For any $\epsilon \in ]0,1[$, there exists a neural network $\widehat{D}_\epsilon\in \cN\cN(W^*_\epsilon,L^*,\beta^*_\epsilon)$ with \begin{align*}
			&W^*_\epsilon \lesssim rR^{d_{\cS}+1}\pa{\frac{\epsilon}{\diam(\cS)}}^{-d_{\cS}},\\
			&L^* = \lceil \log_2(d_{\cS}) \rceil + \lceil \log_2(r) \rceil + c,\\
			&\beta^*_\epsilon \lesssim \sqrt{r}\pa{\frac{\epsilon}{M}}^{-p}
		\end{align*}
		such that \begin{equation}
			\norm{\widehat{D}_\epsilon(\bx)-\bx}\le \epsilon
		\end{equation}
		for all $\bx \in \cS$, where $c\in \N$ is a universal constant and $p>0$ only depends on $\cS$.
	\end{lemma}
	
	Let $\epsilon \in ]0,1[$. Consider the PnP-iteration with a neural network denoiser
	\begin{equation}\label{eq:pnpiterNN}
			\bx_{k+1} = \widehat{D}_\epsilon^N\pa{\bx_k + \gamma B(\by - A \bx_k)},
	\end{equation}
	where $\widehat{D}_\epsilon^N$ is a solution of \eqref{eq:loss} over $\cN\cN(W^*_\epsilon,L^*,\beta^*_\epsilon)$ with $W^*_\epsilon,L^*,\beta^*_\epsilon$ as in Lemma~\ref{lem:NN_existence_uniform}. We can then derive a pointwise recovery bound for the PnP-iteration \eqref{eq:pnpiterNN}.	
	\begin{theorem}\label{thm:recovery_pointwise_NN}
	Assume that \eqref{assum:Xcompact}--\eqref{assum:Xnoise} and \eqref{assum:ranB}--\eqref{assum:manifold} hold and that $\mu_{\bX}$ fulfills \eqref{assum:XV}. Consider the PnP iteration \eqref{eq:pnpiterNN}. If $L_{\widehat{D}_\epsilon^N} \le 1$ or otherwise $\kappa(BA) \le \pa{1-{L^{-2}_{\widehat{D}_\epsilon^N}}}^{-\frac{1}{2}}$, then the iterates of \eqref{eq:pnpiterNN} with $\gamma = \pa{\kappa(BA)\lambda_{\max}(BA)}^{-1} $ obey
		\begin{align}\label{eq:rcvbndNNptw1}
		\norm{\bx_k-\bx} &\lesssim q^{k+1}\norm{\bx_0 - \bx} \nonumber\\
			&+ \frac{1-q^{k}}{1-q}\Bigg(L_{\widehat{D}_\epsilon^N}\gamma\norm{B}\norm{\be} + \pa{\lip{\widehat{D}_\epsilon^N} +\lip{\widehat{D}_\epsilon}}\pa{1+\sqrt{\frac{4\log(N)}{Nr}}}\sqrt{\lambda_{\max}(\widetilde{\Sigma}_{\bUpsilon})Nr} \nonumber\\
			&+ \pa{1+\lip{\widehat{D}_\epsilon^N}}\dist(\bx,\cD_N^\bX) + \sqrt{N}\epsilon \Bigg) ,
		\end{align}
		with probability at least $1-N^{-2}$ over the sampling of $\cD_N^{\bUpsilon}$, where $q = L_{\widehat{D}_\epsilon^N} \sqrt{1-\kappa(BA)^{-2}}\in [0,1[$, and  $r = \dim V$.

		Moreover, if $\mu_\bX$ has a density with respect to the volume measure on $V$ which is bounded away from zero, then for $N$ large enough, it holds with probability at least $1 - (K_2+1) N^{-2}$ over the sampling of $\cD_N^\bX$ and $\cD_N^{\bUpsilon}$ that
		\begin{align}\label{eq:rcvbndNNptw2}
			\norm{\bx_k-\bx} &\lesssim q^{k+1}\norm{\bx_0 - \bx} \notag\\
			&+ \frac{1-q^{k}}{1-q}\Bigg(L_{\widehat{D}_\epsilon^N}\gamma\norm{B}\norm{\be} + \pa{\lip{\widehat{D}_\epsilon^N} +\lip{\widehat{D}_\epsilon}}\pa{1+\sqrt{\frac{4\log(N)}{Nr}}}\sqrt{\lambda_{\max}(\widetilde{\Sigma}_{\bUpsilon})Nr} \nonumber\\
			&+\pa{1+L_{\widehat{D}_\epsilon^N}}K_1 3^{1/r}\pa{\frac{\log(N)}{N}}^{1/r} + \sqrt{N}\epsilon \Bigg) ,
		\end{align}
		 where $K_1, K_2 > 0$ are constants that depend only on $r$ and $M$.
	\end{theorem}

	\begin{remark}
		One can sharpen the recovery bounds in Theorem~\ref{thm:recovery_pointwise_NN} by estimating the Lipschitz constants of $\widehat{D}_\epsilon$ and $\widehat{D}_\epsilon^N$. We therefore use a coarse estimate for ReLU networks. Let $s_l$ be the number of nonzero entries of $A_l$. Since the entries of $A_l$ are bounded by $\beta$, we have $\|A_l\|\le \beta s_l^{1/2}$. Hence, we can use the standard bound \begin{equation}\label{eq:Lipschitz_bound_NN}
			L_D \le \prod\limits_{l =1}^L \norm{A_l} \le \prod\limits_{l=1}^L \beta s_l^{1/2}\le \beta^L \pa{\frac{W}{L}}^\frac{L}{2}.
		\end{equation} 
		Particularizing to the networks considered in Theorem~\ref{thm:recovery_pointwise_NN}, one obtains \begin{equation}\label{eq:Lipschitz_NN_epsilon}
			\max\left\{ L_{\widehat{D}_\epsilon} ,L_{\widehat{D}_\epsilon^N}\right\} \lesssim \pa{\frac{\epsilon}{M}}^{-\rho_1}r^{\rho_2}, 
		\end{equation} 
		with     
		\begin{align*}
			&\rho_1 = \pa{p+\frac{d_{\cS}}{2}}\pa{\lceil \log_2(d_{\cS}) \rceil + \lceil \log_2(r) \rceil + c},\\
			& \rho_2 =  {\lceil \log_2(d_{\cS}) \rceil + \lceil \log_2(r) \rceil + c},
		\end{align*}
		from the asymptotic bounds of the network's hyperparameters. However, for practical application, this bound could be drastically improved. In view of the Lipschitz continuity and cocoercivity - see Proposition~\ref{prop:cocoercive} - of the MMSE denoiser, one could include such constraints in the training of the network. 
	\end{remark}

	\section{Properties of the MMSE denoiser}\label{sec:MMSE_properties}
	
	We show several properties of the MMSE denoiser \eqref{eq:dmmse1} which are the essential tools for proving the main results in Section~\ref{sec:proofs}.
	
	\subsection{Lipschitz continuity }
	Lipschitz stability of the MMSE denoiser is a consequence of the following proposition and the mean value theorem.
	\begin{proposition}\label{prop:jacDmu}
		Assume that \eqref{assum:Xcompact} and \eqref{assum:Xnoise} hold.  Let $\mu \in \cP(\cX)$ satisfy \eqref{assum:XV} then for every $\bz \in V$, the Jacobian is given by \begin{equation}\label{eq:Jacobian_MMSE}
			\jac{\widehat{D}_{\mu}}(\bz) = \Sigma_{\bX|\bz}\Pi\widetilde{\Sigma}_{\bUpsilon}^{-1}\Pi^\top
		\end{equation}
		where $\Sigma_{\bX|\bz}$ is the conditional variance-covariance matrix, i.e.,
		\[
		\Sigma_{\bX|\bz} = 
		\frac{\int_{\cX} \bx\bx^\top \varphi(\bz - \bx;\widetilde{\Sigma}_{\bUpsilon}) d\mu(\bx)}{p_\mu(\bz)} - \widehat{D}_{\mu}(\bz)\widehat{D}_{\mu}(\bz)^\top .
		\]
		Moreover, $\jac{\widehat{D}_{\mu}}(\bz)$ is diagonalizable, has nonnegative eigenvalues, and
		\[
		\sup_{\bz \in V} \norm{\jac{\widehat{D}_{\mu}}(\bz)} \leq \frac{M^2}{\lambda_{\min}(\widetilde{\Sigma}_{\bUpsilon})},
		\]
		where $M>0$ is an upper bound of $\norm{\bx}$ for all $\bx \in \cX$.
		\begin{proof}
			By \eqref{assum:Xcompact} and standard arguments (differentiability of the integrand and Lebesgue's dominated convergence theorem), one can interchange the differentiation and integration. Therefore, it holds
			\begin{align}
				\jac{\widehat{D}_{\mu}}(\bz) 
				&= \int_{\cX} \bx\nabla_{\bz}{\pa{\frac{\varphi(\bz - \bx;\widetilde{\Sigma}_{\bUpsilon})}{p_\mu(\bz)}}}^\top d\mu(\bx) \nonumber\\
				&= \frac{\int_{\cX} \bx\nabla_{\bz}{\varphi(\bz - \bx;\widetilde{\Sigma}_{\bUpsilon})}^\top d\mu(\bx)}{p_\mu(\bz)} - \frac{\int_{\cX} \bx\varphi(\bz - \bx;\widetilde{\Sigma}_{\bUpsilon})d\mu(\bx)}{p_\mu(\bz)} \frac{\nabla_\bz p_\mu(\bz)}{p_\mu(\bz)}^\top \nonumber\\
				&= \frac{\int_{\cX} \bx(\bx-\bz)^\top\Pi\widetilde{\Sigma}_{\bUpsilon}^{-1}\Pi^\top\varphi(\bz-\bx;\widetilde{\Sigma}_{\bUpsilon})d\mu(\bx)}{p_\mu(\bz)} - \widehat{D}_{\mu}(\bz)\nabla_\bz \log(p_\mu(\bz))^\top \nonumber\\
				&= \frac{\int_{\cX} \bx \bx^\top \varphi(\bz - \bx;\widetilde{\Sigma}_{\bUpsilon})d\mu(\bx)}{p_\mu(\bz)}\Pi\widetilde{\Sigma}_{\bUpsilon}^{-1}\Pi^\top - \widehat{D}_{\mu}(\bz) \bz^\top\Pi\widetilde{\Sigma}_{\bUpsilon}^{-1}\Pi^\top - \widehat{D}_{\mu}(\bz) \nabla_\bz \log(p_\mu(\bz))^\top \nonumber\\
				&= \frac{\int_{\cX} \bx \bx^\top \varphi(\bz - \bx;\widetilde{\Sigma}_{\bUpsilon})d\mu(\bx)}{p_\mu(\bz)}\Pi\widetilde{\Sigma}_{\bUpsilon}^{-1}\Pi^\top - \widehat{D}_{\mu}(\bz)\pa{\Pi\widetilde{\Sigma}_{\bUpsilon}^{-1}\Pi^\top\bz+\nabla_\bz \log(p_\mu(\bz))}^\top . \label{eq:jacproof}
			\end{align}
			With a similar derivation as above, one shows that
			\begin{equation}\label{eq:pretre}
				\nabla_\bz \log(p_\mu(\bz)) = \Pi\widetilde{\Sigma}_{\bUpsilon}^{-1}\Pi^\top\pa{\widehat{D}_{\mu}(\bz) - \bz} .
			\end{equation}
			Inserting this into \eqref{eq:jacproof}, we get \eqref{eq:Jacobian_MMSE}. \\
			The fact that $J_{\widehat{D}_{\mu}}(\bz)$ is diagonalizable with nonnegative eigenvalues follows from \cite[Corollary~7.6.2]{Horn13}. In addition, for any $\bz \in V$, it holds
			\[
			\norm{\jac{\widehat{D}_{\mu}}(\bz)} \leq \norm{\Pi\widetilde{\Sigma}_{\bUpsilon}^{-1}\Pi^\top}\norm{\Sigma_{\bX|\bz}} = \lambda_{\min}(\widetilde{\Sigma}_{\bUpsilon})^{-1}\norm{\Sigma_{\bX|\bz}} .
			\]
			Let $\bs$ as an arbitrary vector on the unit sphere $\bbS^{r-1}$ of $V$. We have
			\[
			0 \leq \dotp{\Sigma_{\bX|\bz}\bs}{\bs} 
			\leq \frac{\int_{\cX} |\dotp{\bx}{\bs}|^2 \varphi(\bz - \bx;\widetilde{\Sigma}_{\bUpsilon}) d\mu(\bx)}{p_\mu(\bz)} \leq \frac{\int_{\cX} \norm{\bx}^2 \varphi(\bz - \bx;\widetilde{\Sigma}_{\bUpsilon}) d\mu(\bx)}{p_\mu(\bz)} \leq M^2 .
			\]
			Taking the supremum over $\bs$ and $\bz$ concludes the proof.
		\end{proof}
	\end{proposition}

	\begin{remark}\label{rem:tde} \, \begin{enumerate}[label = (\roman*)]
			\item The identity \eqref{eq:Jacobian_MMSE} has a long history and has had several applications in information and estimation theory. When $\Sigma_{\bUpsilon}$ is positive definite, i.e., $V=\bbR^n$, this identity was proved in \cite{Guo05,Palomar06}; see also the general version that appeared recently in \cite{Dytso20}. It was used in \cite{Guo11} to study certain regularity, monotonicity, and convexity properties of the MMSE risk as a function of the SNR. The special case $\Sigma_{\bUpsilon} = I_n$ is due to \cite{Hatsell71}, while the scalar case was considered in \cite{Wu12} to establish some regularity properties of the MMSE risk and mutual information. 
			\item Since $\Pi^\top\Pi=I_r$, $\Pi\Pi^\top$ is the orthogonal projector on $V$, the identity \eqref{eq:pretre} is equivalent to
			\begin{equation}\label{eq:tre}
				\widehat{D}_{\mu}(\bz) = \bz + \Pi\widetilde{\Sigma}_{\bUpsilon}\Pi^\top \nabla \log(p_\mu(\bz)) ,
			\end{equation}
			due to $\bz,\widehat{D}_{\mu}(\bz),\nabla_\bz \log(p_\mu)(\bz)\in V$. Note that the term $\nabla \log(p_\mu(\bz))$ is commonly known as the score function. Further, the identity \eqref{eq:tre} is reminiscent of Tweedie's formula \eqref{eq:Tweedie}.
		\end{enumerate}
		
	\end{remark}
	
	As an immediate consequence of Proposition~\ref{prop:jacDmu}, we obtain the following Lemma. 
	\begin{lemma}\label{lem:Dmulip}
		Assume that \eqref{assum:Xcompact} and \eqref{assum:Xnoise} hold and denote $M>0$ an upper bound of $\norm{\bx}$ for $\bx \in \cX$. For any measure $\mu \in \cP(\cX)$ satisfying \eqref{assum:XV}, $\widehat{D}_{\mu}$ is $\lip{\widehat{D}_{\mu}}$-Lipschitz continuous on $V$ with $\lip{\widehat{D}_{\mu}} \leq \frac{M^2}{\lambda_{\min}(\widetilde{\Sigma}_{\bUpsilon})}$. 
	\end{lemma}
	The bound for the Lipschitz constant in Lemma~\ref{lem:Dmulip} may be improved on a case by case basis. Further assumption \eqref{assum:Xcompact} and thus the existence of an upper bound $M$ is sufficient but not necessary for the Lipschitz continuity of $\widehat{D}_{\mu}$. For example, if $\bX$ is a Gaussian random vector with covariance matrix $\Sigma_{\bX}$, it holds $\widehat{D}_{\mu_\bX}(\bz) = \Sigma_\bX(\Sigma_\bX+\Sigma_{\bUpsilon})^+\bz$. Consequently $\widehat{D}_{\mu_\bX}$ is Lipschitz continuous with Lipschitz constant $\norm{\Sigma_\bX(\Sigma_\bX+\Sigma_{\bUpsilon})^+}$. For the empirical MMSE denoiser $\widehat{D}_{\mu_\bX^N}$ the Lipschitz constant can be bounded using concetration inequalities. We leave the details to a future work. 
	
	The main advantage of the above Lipschitz bound, however, is its independence of the measure $\mu$. Moreover, this bound can be interpreted as an upper-bound on the signal-to-noise ratio (SNR). It also reflects that this constant degrades as the noise decreases. An explanation for this behavior can be seen in Proposition~\ref{prop:limsigma}, which shows that $\widehat{D}_{\mu}$ can become set-valued as $\Sigma_{\bUpsilon} \to 0$.
	
	\subsection{Stability with respect to the prior measure}
	\begin{theorem}\label{thm:lipDmmseW2}
		Assume that \eqref{assum:Xcompact} and \eqref{assum:Xnoise} hold and that the measures $\mu, \nu \in \cP(\cX)$ satisfy \eqref{assum:XV}. Let $\bUpsilon$ as in \eqref{assum:Xnoise}, and $\bU \sim \mu$, $\bV \sim \nu$ independent of $\bUpsilon$ . Denote $\widehat{\mu}_{\mmse}$ (resp. $\widehat{\nu}_{\mmse}$) the law of $\widehat{D}_{\mu}(\bU+\bUpsilon)$ (resp. $\widehat{D}_{\nu}(\bV+\bUpsilon)$). Then the following estimates hold.
		\begin{enumerate}[label=(\roman*)]
			\item \label{thm:lipDmmseW2_claim1}
			$
			\norm{\Expect{}{\widehat{D}_{\mu}(\bU+\bUpsilon)} - \Expect{}{\widehat{D}_{\nu}(\bV+\bUpsilon)}} \leq \cW_1(\mu,\nu) .
			$
			
			\item \label{thm:lipDmmseW2_claim2}
			$
			\abs{\cR(\widehat{D}_{\mu}) - \cR(\widehat{D}_{\nu})} \leq 2 \sqrt{\lambda_{\max}(\widetilde{\Sigma}_{\bUpsilon}) r} \pa{1+\frac{M^2}{\lambda_{\min}(\widetilde{\Sigma}_{\bUpsilon})}} \cW_2(\mu,\nu)
			$, where $\cR$ is as in \eqref{eq:mmse_objective} and ${r=\rank\pa{\Sigma_\Upsilon}}$.
			
			\item \label{thm:lipDmmseW2_claim3}
			For $p \in [1,2]$ it holds
			\[
			\cW_p(\widehat{\mu}_{\mmse},\widehat{\nu}_{\mmse}) \leq 2\sqrt{\lambda_{\max}(\widetilde{\Sigma}_{\bUpsilon}) r} + \cW_p(\mu,\nu) .
			\]
		\end{enumerate}
		\begin{proof} \,
			\begin{enumerate}[label=(\roman*)]
				\item By assumption \eqref{assum:Xnoise}, the densities of $\bU + \bUpsilon$ and $ \bV+\bUpsilon$ with respect to the $r$-dimensional Hausdorff measure are given by \begin{equation*}
					\int_{\cX}\phi(\bz-\bu) { d\mu(\bu)} \,\text{ and }    \int_{\cX}\phi(\bz-\bv) {d\nu(\bv)},
				\end{equation*}
				respectively. Thus, we have
				\begin{align}\label{eq:difEDmueq}
					&\Expect{}{\widehat{D}_{\mu}(\bU+\bUpsilon)} - \Expect{}{\widehat{D}_{\nu}(\bV+\bUpsilon)} \nonumber\\
					&= (2\pi)^{-r/2}\det(\widetilde{\Sigma}_{\bUpsilon})^{-1/2}\pa{\int_{V}\int_{\cX} \bu \varphi(\bz - \bu;\widetilde{\Sigma}_{\bUpsilon}) d\mu(\bu)d\cH^r(\bz) - \int_{V}\int_{\cX} \bv \varphi(\bz' - \bv;\widetilde{\Sigma}_{\bUpsilon}) d\nu(\bv)d\cH^r(\bz)} \nonumber\\
					&= (2\pi)^{-r/2}\det(\widetilde{\Sigma}_{\bUpsilon})^{-1/2}\pa{\int_{\cX} \int_{V} \bu \varphi(\bz - \bu;\widetilde{\Sigma}_{\bUpsilon}) d\cH^r(\bz)d\mu(\bu) - \int_{\cX} \int_{V} \bv \varphi(\bz' - \bv;\widetilde{\Sigma}_{\bUpsilon}) d\cH^r(\bz')d\nu(\bv)} \nonumber\\
					&= \int_{\cX} \bu d\mu(\bu) - \int_{\cX} \bv d\nu(\bv) , 
				\end{align}
				where, in the second equality, we used \eqref{assum:Xcompact} and Fubini's theorem, \eqref{assum:XV} and standard integration results for Gaussian distributions. Fix $\bs$ as an arbitrary vector on the unit sphere $\bbS^{r-1}$ of $V$. We then have for all $\bu, \bv \in \cX$
				\begin{align*}
					\abs{\dotp{\bs}{\bu} - \dotp{\bs}{\bv}} \leq \norm{\bu-\bv} .
				\end{align*}
				Consequently, taking the inner product with $\bs \in \bbS^{r-1}$ in \eqref{eq:difEDmueq}, we get for any coupling $\pi$ of $\bU$ and $\bV$ that
				\begin{align*}
					\norm{\Expect{}{\widehat{D}_{\mu}(\bU+\bUpsilon)} - \Expect{}{\widehat{D}_{\nu}(\bV+\bUpsilon)}} 
					&= \sup_{\bs \in \bbS^{r-1}} \abs{\dotp{\bs}{\Expect{}{\widehat{D}_{\mu}(\bU+\bUpsilon)} - \Expect{}{\widehat{D}_{\nu}(\bV+\bUpsilon)}}} \\
					&= \sup_{\bs \in \bbS^{r-1}} \abs{\int_{\cX} \dotp{\bs}{\bu} d\mu(\bu) - \int_{\cX} \dotp{\bs}{\bv} d\nu(\bv)} \\
					&= \sup_{\bs \in \bbS^{r-1}} \abs{\int_{\cX}\dotp{\bs}{\bu-\bv}d\pi(\bu,\bv)}\\
					&\le \int_{\cX} \norm{\bu-\bv}d\pi(\bu,\bv) .
				\end{align*}
				Taking the infimum with respect to the coupling yields the claim.
					%
					%
				
				\item Let $\bU$ and $\bV$ be jointly distributed according to their optimal coupling so that
				\[
				\cW_2(\mu,\nu) = \Expect{}{\norm{\bU - \bV}^2}^{1/2} .
				\]
				We have from the optimality of the MMSE estimator (see \eqref{eq:mmse}) that
				\begin{align*}
					\sqrt{\cR(\widehat{D}_{\mu})} 
					&\leq \Expect{}{\norm{\widehat{D}_{\nu}(\bU+\bUpsilon)-\bU}^2}^{1/2} \\
					&\leq \Expect{}{\norm{\widehat{D}_{\nu}(\bV+\bUpsilon)-\bV}^2}^{1/2} + \Expect{}{\norm{\widehat{D}_{\nu}(\bU+\bUpsilon)-\widehat{D}_{\nu}(\bV+\bUpsilon)}^2}^{1/2} + \Expect{}{\norm{\bU-\bV}^2}^{1/2} \\
					&= \sqrt{\cR(\widehat{D}_{\nu})} + \Expect{}{\norm{\widehat{D}_{\nu}(\bU+\bUpsilon)-\widehat{D}_{\nu}(\bV+\bUpsilon)}^2}^{1/2} + \cW_2(\mu,\nu) .
				\end{align*}
				For the second term, we use the mean value theorem and Lemma~\ref{lem:Dmulip} to get that
				\[
				\Expect{}{\norm{\widehat{D}_{\nu}(\bU+\bUpsilon) - \widehat{D}_{\nu}(\bV+\bUpsilon)}^2}^{1/2} 
				\leq \frac{M^2}{\lambda_{\min}(\widetilde{\Sigma}_{\bUpsilon})} \Expect{}{\norm{\bU - \bV}^2}^{1/2}
				= \frac{M^2}{\lambda_{\min}(\widetilde{\Sigma}_{\bUpsilon})} \cW_2(\mu,\nu) .
				\]
				Reverting the role of $\mu$ and $\nu$, we have the same bound. We therefore get
				\[
				\abs{\sqrt{\cR(\widehat{D}_{\mu})}  - \sqrt{\cR(\widehat{D}_{\nu})}} \leq \pa{1+\frac{M^2}{\lambda_{\min}(\widetilde{\Sigma}_{\bUpsilon})}}\cW_2(\mu,\nu) .
				\]
				Observe now that from the optimality of the MMSE estimator again, we have
				\begin{equation}\label{eq:mmseoptbnd}
					\cR(\widehat{D}_{\mu}) \leq \Expect{}{\norm{(\bU+\bUpsilon)-\bU}^2}
					= \Expect{}{\norm{\bUpsilon}^2} = \Expect{}{\norm{\Pi^\top\bUpsilon}^2} = \tr{\widetilde{\Sigma}_{\bUpsilon}}
					\leq \lambda_{\max}(\widetilde{\Sigma}_{\bUpsilon}) r ,
				\end{equation}
				and similarly for $\cR(\widehat{D}_{\nu})$. It then follows that
				\begin{align*}
					\abs{\cR(\widehat{D}_{\mu}) - \cR(\widehat{D}_{\nu})} 
					&= \pa{\sqrt{\cR(\widehat{D}_{\mu})}  + \sqrt{\cR(\widehat{D}_{\nu})}}\abs{\sqrt{\cR(\widehat{D}_{\mu})}  - \sqrt{\cR(\widehat{D}_{\nu})}} \\
					&\leq 2 \sqrt{\lambda_{\max}(\widetilde{\Sigma}_{\bUpsilon}) r} \pa{1+\frac{M^2}{\lambda_{\min}(\widetilde{\Sigma}_{\bUpsilon})}} \cW_2(\mu,\nu)
				\end{align*}
				as claimed.
				
				\item The triangle inequality yields
				\begin{equation}\label{eq:triangleWpmmse}
					\cW_p(\widehat{\mu}_{\mmse},\widehat{\nu}_{\mmse}) \leq \cW_p(\widehat{\mu}_{\mmse},\mu) + \cW_p(\widehat{\nu}_{\mmse},\nu) + \cW_p(\mu,\nu) .
				\end{equation}
				We have the bound 
				\begin{align*}
					\cW_p(\widehat{\mu}_{\mmse},\mu)^2 
					\leq \cW_2(\widehat{\mu}_{\mmse},\mu)^2
					&\leq \Expect{}{\norm{\widehat{D}_{\mu}(\bU+\bUpsilon)-\bU}^2} \\
					&= \cR(\widehat{D}_{\mu}) \leq \Expect{}{\norm{\bUpsilon}^2} \leq \lambda_{\max}(\widetilde{\Sigma}_{\bUpsilon})r .
				\end{align*}
				In the first inequality, we used \eqref{eq:Wporder}, in the second one we used optimality in the definition of the Wasserstein distance, and the third inequality follows from \eqref{eq:mmseoptbnd}. The same bound holds for $\cW_p(\widehat{\nu}_{\mmse},\nu)$. Inserting this into \eqref{eq:triangleWpmmse} proves the result.
			\end{enumerate}
		\end{proof}
	\end{theorem}
	
	\begin{remark}
		The special scalar case of Theorem~\ref{thm:lipDmmseW2}\ref{thm:lipDmmseW2_claim2} was considered in \cite[Theorem~5]{Wu12}. Moreover, the constant $M^2/\lambda_{\min}(\widetilde{\Sigma}_{\bUpsilon})$ in Theorem~\ref{thm:lipDmmseW2}\ref{thm:lipDmmseW2_claim2} can be replaced by $\max(\lip{\widehat{D}_{\mu}},\lip{\widehat{D}_{\nu}})$. 
	\end{remark}
	
	A useful application of this result is when $\nu$ is the empirical version of $\mu$ as discussed in Section~\ref{sec:mmsedenoiser}; see \eqref{eq:dmmsegauss} and \eqref{eq:Ndmmsegauss}.

Let $\mu^N$ be the empirical measure obtained from $N$ independent samples from $\mu$. Let $\widehat{\mu}^N_{\mmse}$ be the law of $\widehat{D}_{\mu^N}(\bV+\bUpsilon)$, where $\bV \sim \mu^N$. We have the following corollary.
\begin{corollary}\label{cor:lipDmmseempiricalW2}
Assume that \eqref{assum:Xcompact} and \eqref{assum:Xnoise} hold and that $\mu \in \cP(\cX)$ satisfies \eqref{assum:XV}.
\begin{enumerate}[label=(\roman*)]
\item \label{cor:lipDmmseempiricalW2_claim1}
If $s > d_{\mathrm{M}}(\mu) \geq 2$, then
\[
\Expect{}{\norm{\Expect{}{\widehat{D}_{\mu}(\bU+\bUpsilon)} - \Expect{}{\widehat{D}_{\mu^N}(\bV+\bUpsilon)}}} \lesssim \diam(\cX) N^{-1/s} .
\]

\item \label{cor:lipDmmseempiricalW2_claim2}
If $s > d_{\mathrm{M}}(\mu) \geq 4$, then
\[
\Expect{}{\abs{\cR(\widehat{D}_{\mu}) - \cR(\widehat{D}_{\mu^N})}} \lesssim \sqrt{\lambda_{\max}(\widetilde{\Sigma}_{\bUpsilon}) r} \pa{1+\frac{M^2}{\lambda_{\min}(\widetilde{\Sigma}_{\bUpsilon})}} \diam(\cX) N^{-1/s} .
\]

\item \label{cor:lipDmmseempiricalW2_claim3}
For $p \in [1,2]$, if $s > d_{\mathrm{M}}(\mu) \geq 2p$, then
\[
\Expect{}{W_p(\widehat{\mu}_{\mmse},\widehat{\mu}^N_{\mmse})} \lesssim \sqrt{\lambda_{\max}(\widetilde{\Sigma}_{\bUpsilon}) r} + \diam(\cX) N^{-1/s} .
\]
\end{enumerate}
The (outer) expectation is with respect to the sampling of the atoms in $\mu^N$. 
\end{corollary}

\begin{remark}\label{rem:Wdiscrete}
The scaling in $N$ in the above bounds can be sharpened with a finer notion of dimension; see \cite[Theorem~1]{Weed19} where the authors also provide lower-bounds showing that the scaling is essentially tight in general. This can only improved if additional regularity assumptions are imposed. If $\mu$ is supported on a compact $d$-dimensional differentiable manifold and is absolutely continuous wrt the $d$-dimensional Hausdorff measure on it, then $d_{\mathrm{M}}(\mu_{\bX})=d$ and $s$ in the above bounds can be replaced by $d$; see \cite[Proposition~8 and 9]{Weed19}.
\end{remark}

\begin{proof}
It is sufficient to bound $\Expect{}{W_p(\mu,\mu^N)}$ and then invoke Theorem~\ref{thm:lipDmmseW2}. This bound follows from \cite[Corollary~1.2]{Boissard14} for $s$ as devised.
\end{proof}

	\subsection{Monotonicity properties}
	Monotonicity properties of the MMSE denoiser can be derived using Proposition~\ref{prop:jacDmu}.\\
	
	\begin{proposition}\label{prop:Dmumonotone}  
		Assume that \eqref{assum:Xcompact} and \eqref{assum:Xnoise} hold. Let $\mu\in\mathcal P(\mathcal X)$ satisfy \eqref{assum:XV}, and set $U:=\Pi\widetilde\Sigma_\Upsilon^{-1}\Pi^\top$. Then $\widehat D_\mu$ is monotone and $\frac{M^2}{\lambda_{\min}(\widetilde\Sigma_\Upsilon)}$-Lipschitz continuous on the Hilbert space $(V,\|\cdot\|_U)$. Moreover, $\log p_\mu$ is $\lambda_{\min}(\widetilde\Sigma_\Upsilon)^{-1}$-weakly convex on $(V,\norm{\cdot})$.
		\begin{proof}
			For any $\bu, \bv \in V$, we have
			\begin{align*}
				\dotp{\widehat{D}_\mu(\bu)-\widehat{D}_\mu(\bv)}{\bu-\bv}_U 
				&= \int_0^1 \dotp{\jac{\widehat{D}_{\mu}}(\bv+t(\bu-\bv))(\bu-\bv)}{\bu-\bv}_U dt \\
				&= \int_0^1 \dotp{\Sigma_{\bX|\bv+t(\bu-\bv)}U(\bu-\bv)}{\bu-\bv}_U dt \\
				&= \int_0^1 \dotp{U\Sigma_{\bX|\bv+t(\bu-\bv)}U(\bu-\bv)}{\bu-\bv} dt \\
				&= \int_0^1 \dotp{\Sigma_{\bX|\bv+t(\bu-\bv)}U(\bu-\bv)}{U(\bu-\bv)} dt \geq 0 .
			\end{align*}
			In the second equality, we used the derivative identity of Proposition~\ref{prop:jacDmu}, and in the last inequality that the covariance matrix $\Sigma_{\bX|\bz}$ is symmetric, semidefinite positive and that $U$ is symmetric and positive definite on $V$.\\
			Now, with similar arguments, we also have
			\begin{align*}
				\norm{\widehat{D}_\mu(\bu)-\widehat{D}_\mu(\bv)}_U
				&\leq \int_0^1 \norm{U^{1/2}\jac{\widehat{D}_{\mu}}(\bv+t(\bu-\bv))(\bu-\bv)} dt \\
				&= \int_0^1 \norm{U^{1/2}\Sigma_{\bX|\bv+t(\bu-\bv)}U(\bu-\bv)} dt \\
				&\leq \int_0^1 \norm{U^{1/2}\Sigma_{\bX|\bv+t(\bu-\bv)}U^{1/2}}dt ~ \norm{\bu-\bv}_U \\
				&\leq \frac{M^2}{\lambda_{\min}(\widetilde{\Sigma}_{\bUpsilon})}\norm{\bu-\bv}_U .
			\end{align*}\\
			Weak convexity is equivalent to hypomonotoncity of the score $\nabla \log p_\mu$ on $V$. The latter holds by monotonicity of $\widehat{D}_\mu$ on $V_U$ and Tweedie's formula \eqref{eq:tre},
			\begin{align*}
				&\lambda_{\min}(\widetilde{\Sigma}_{\bUpsilon})^{-1}\norm{\bu -\bv}^2 + \dotp{\nabla \log p_\mu(\bu)-\nabla \log p_\mu(\bv)}{\bu - \bv} \\
				&\geq \norm{\bu -\bv}_U^2 + \dotp{U^{-1}\nabla \log p_\mu(\bu)-U^{-1}\nabla \log p_\mu(\bv)}{\bu - \bv}_U = \dotp{\widehat{D}_\mu(\bu) - \widehat{D}_\mu(\bv)}{\bu - \bv}_U \geq 0 .
			\end{align*}

		\end{proof}
	\end{proposition}
	
	\begin{remark}\label{rem:stron_monotonicity}
		If $\Sigma_{\bX|\bz}$ is uniformly positive definite in $z$ on V, i.e. there exists $\alpha >0$ such that \begin{equation*}
			\dotp{\Sigma_{\bX|\bz}w}{w} \ge \alpha
		\end{equation*}
		for all $w\in V$ with $\norm{w}=1$, then the MMSE estimsator is $ \alpha \lambda_{\max}(\widetilde{\Sigma}_{\bUpsilon})^{-1}$-strongly monotone with respect to $\norm{\cdot}_U$. Consequently, $\log p_\mu$ is $\pa{\lambda_{\min}(\widetilde{\Sigma}_{\bUpsilon})^{-1} - \alpha \lambda_{\max}(\widetilde{\Sigma}_{\bUpsilon})^{-1}}$-weakly convex with respect to the standard euclidean norm. In particular, this implies strong convexity for $\alpha < \frac{\lambda_{\max}(\widetilde{\Sigma}_{\bUpsilon})}{\lambda_{\min}(\widetilde{\Sigma}_{\bUpsilon})} = \kappa\pa{\widetilde{\Sigma}_{\bUpsilon}}$.
	\end{remark}

	In the special case where $\widetilde{\Sigma}_{\bUpsilon}  = \sigma_{\bUpsilon}^2 I_r$, even more can be said due to symmetry. We recall that an operator $T: V \to V$ is $\beta$-cocoercive, $\beta > 0$, if
	\[
	\dotp{T(\bx) - T(\bz)}{\bx - \bz} \geq \beta\norm{T(\bx)-T(\bz)}^2, \qforallq \bx, \bz \in V .
	\]
	$\beta$-cocoercivity implies $1/\beta$-Lipschitz continuity. The converse is not true in general, unless $T$ is the gradient of a differentiable convex function, a fact known as the Baillon-Haddad theorem.
	\begin{proposition}\label{prop:cocoercive}
		Assume that \eqref{assum:Xcompact} holds, that $\widetilde{\Sigma}_{\bUpsilon}  = \sigma_{\bUpsilon}^2 I_r$, $\sigma_{\bUpsilon} > 0$, and that $\mu \in \cP(\cX)$ satisfies \eqref{assum:XV}. Then, the operator $\widehat{D}_\mu$ is $\lip{\widehat{D}_\mu}^{-1}$-cocoercive on $V$ with $\lip{\widehat{D}_\mu} \leq \frac{M^2}{\sigma_{\bUpsilon}^2}$.
		\begin{proof}
			In this case, $U=\sigma_{\bUpsilon}^{-2}\Pi \Pi^\top$ and thus $\dotp{\bx}{\bz}_U = \sigma_{\bUpsilon}^{-2}\dotp{\bx}{\bz}$ for all $\bx,\bz \in V$. Monotonicity on $V$ then follows from Proposition~\ref{prop:Dmumonotone} since
			\[
			\dotp{\widehat{D}_\mu(\bu)-\widehat{D}_\mu(\bv)}{\bu-\bv}
			= \sigma_{\bUpsilon}^{-2}\int_0^1 \dotp{\Sigma_{\bX|\bv+t(\bu-\bv)}(\bu-\bv)}{\bu-\bv} dt \geq 0 , \qforallq \bu,\bv \in V .
			\]
			From Lemma~\ref{lem:Dmulip}, $\widehat{D}_\mu$ is also $\lip{\widehat{D}_\mu}$-Lipschitz continuous on $V$ with $\lip{\widehat{D}_\mu} \leq \frac{M^2}{\sigma_{\bUpsilon}^2}$. Moreover, Tweedie's formula \eqref{eq:pretre} now reads,
			\begin{equation*}
				\widehat{D}_{\mu}(\bz) = \bz + \sigma_{\bUpsilon}^2 \nabla \log(p_\mu(\bz)) , \qforallq \bz \in V ,
			\end{equation*}
			that is, $\widehat{D}_{\mu}$ is the gradient of the function $\bz \in V \mapsto h(\bz) \eqdef \frac{1}{2}\norm{\bz}^2 + \sigma_{\bUpsilon}^2 \log(p_\mu(\bz))$. Thus monotonicity of $\widehat{D}_{\mu}$ is equivalent to convexity of $h$; see \cite[Proposition~17.7]{CombettesBauschke17}. It follows from the Baillon-Haddad theorem that $\widehat{D}_{\mu}$ is $\lip{\widehat{D}_\mu}^{-1}$-cocoercive; see \cite[Corollary~18.17]{CombettesBauschke17}.
		\end{proof}
	\end{proposition}
	
	To the best of our knowledge, the cocoercivity of the MMSE estimator has not been shown before. If $\lip{\widehat{D}_\mu}\le 1$, then $\widehat{D}_\mu$ is firmly nonexpansive. This property has recently been of interest in a number of studies, on the one side as a constraint for training a learned denoiser \cite{Terris_2020, BrediesLearningFirmlyNonexp}\cite{nair2024averaged}, but also as a tool in the convergence analysis of PnP methods \cite{Sun_2021, xu2020provable}. While the focus of these works is to ensure firm nonexpansiveness to obtain theoretical convergence guarantees, Proposition~\ref{prop:cocoercive} gives an additional justification for enforcing it, being a property of the MMSE denoiser.

	\subsection{Asymptotic behavior of the MMSE denoiser}\label{subsec:mmselimit}
	We conclude with an investigation of the asymptotic behavior of the MMSE denoiser as the noise $\bUpsilon$ vanishes. For simplicity, we focus on the case $\widetilde{\Sigma} = \sigma_{\bUpsilon}^2 I_r$. The general result can be found in appendix~\ref{append:ssymptotic}. For $\bz \in V$, let
	\[
	d(\bz) \eqdef \dist(\bz,\supp(\mu)) = \inf_{\bx\in \supp(\mu)} \norm{\bx-\bz}
	\]
	and
	\[
	\proj{\supp(\mu)}(\bz) \eqdef \enscond{\bx \in \supp(\mu)}{\norm{\bx-\bz} = d(\bz)}.
	\] 
	$\proj{\supp(\mu)}: V \to 2^{\supp(\mu)}$ is the set-valued orthogonal projector onto $\supp(\mu)$. To emphasize the dependence of the MMSE estimator on $\sigma_{\bUpsilon}$, we will use the notation $\widehat{D}^{\sigma_{\bUpsilon}}_{\mu}$ instead of $\widehat{D}_{\mu_\bX}$.
	
	\begin{proposition}\label{prop:limsigma}
		Assume that \eqref{assum:Xcompact} holds,  the covariance of $\bUpsilon$ is given by $\widetilde{\Sigma}_{\bUpsilon}  = \sigma_{\bUpsilon}^2 I_r$ with $\sigma_{\bUpsilon} > 0$, and that $\mu \in \cP(\cX)$ satisfies \eqref{assum:XV}. For $\bz \in V$ and a $\mu$-measurable set $\cA$, define 
		\[
		\mu_{\bz}^{\sigma_{\bUpsilon}}(\cA) = \frac{\int_{\cA} e^{-\frac{\norm{\bx-\bz}^2}{2\sigma_{\bUpsilon}^2}} d\mu(\bx)}{\int_{\cX} e^{-\frac{\norm{\bx-\bz}^2}{2\sigma_{\bUpsilon}^2}} d\mu(\bx)} .
		\]
		Then the following hold:
		\begin{enumerate}[label=(\roman*)]
			\item \label{prop:limsigma_claim1}
			$(\mu_{\bz}^{\sigma_{\bUpsilon}})_{\sigma_{\bUpsilon}>0}$ subsequentially converges with respect to the Wasserstein distance $\cW_p$ as $\sigma_{\bUpsilon} \to 0$ and each cluster point $\bar{\mu}_{\bz}$ satisfies
			\[
			\supp(\bar{\mu}_{\bz}) \subset \proj{\supp(\mu)}(\bz).
			\]
			\item \label{prop:limsigma_claim2}
			For all $\bz \in V$,
			\[
			\dist(\widehat{D}^{\sigma_{\bUpsilon}}_{\mu}(\bz),\widehat{D}^{0}(\bz)) \to 0 \qasq \sigma_{\bUpsilon} \to 0 ,
			\]
			where $\widehat{D}^{0}: V \to 2^{\conv{\supp(\mu)}}$ is the maximal monotone operator defined by
			\[
			\widehat{D}^{0}(\bz) = \conv{\proj{\supp(\mu)}(\bz)} \subset \conv{\supp(\mu)} .
			\]
		\end{enumerate}
		\begin{proof}  \,\\ 
			\ref{prop:limsigma_claim1} Let $h_{\bz}(\bx) = \norm{\bx-\bz}$. Since $\cX$ is compact by \eqref{assum:Xcompact}, and $\supp(\mu)$ is closed by definition, it is compact. Thus, $h_{\bz}$ attains its minimum on it, with value $d(\bz)$, and $\proj{\supp(\mu)}(\bz)$ is nonempty and compact. Moreover, as $h_{\bz}$ is continuous on $\supp(\mu)$, it is uniformly continuous there and so is $h_{\bz}^2$. \\
			Given $\epsilon > 0$, let
			\[
			\cU_{\bz}^{\epsilon} = \enscond{\bx \in \supp(\mu)}{h_\bz^2(\bx) \leq d^2(\bz) + \epsilon} .
			\] 
			By uniform continuity, there exists $\delta_\epsilon > 0$ such that for all $x^\star \in \proj{\supp(\mu)}(\bz)$,
			\[
			\supp(\mu) \cap \Ball(\bx^\star,\delta_\epsilon) \subset \cU_{\bz}^{\epsilon} .
			\]
			Let
			\[
			\cK_\bz^\epsilon = \supp(\mu)\cap\pa{\bigcup_{x^\star \in \proj{\supp(\mu)}(\bz)} \Ball(\bx^\star,\delta_\epsilon)} .
			\]
			This is a relatively open neighborhood of $\proj{\supp(\mu)}(\bz)$ in $\supp(\mu)$, and uniform continuity implies that
			\[
			\cK_\bz^\epsilon \subset \cU_{\bz}^{\epsilon} .
			\]
			Note that due to the separability of $\cX$, it holds for any open set $\cO$ that $\mu(\cO\cap\supp(\mu)^{c}) =0$. Thus, by definition of the support, we have $\mu(\cK_{\bz}^\epsilon) > 0$. Therefore,
			\begin{align*}
				\mu_{\bz}^{\sigma_{\bUpsilon}}(\supp(\mu) \setminus \cU_{\bz}^\epsilon) 
				&= \frac{\int_{\supp(\mu) \setminus \cU_{\bz}^{\epsilon}} e^{-\frac{h_\bz^2(\bx)-d^2(\bz)}{2\sigma_{\bUpsilon}^2}} d\mu(\bx)}{\int_{\cX} e^{-\frac{h_\bz^2(\bx)-d^2(\bz)}{2\sigma_{\bUpsilon}^2}} d\mu(\bx)} \\
				&\leq \frac{\int_{\supp(\mu) \setminus \cU_{\bz}^{\epsilon}} e^{-\frac{h_\bz^2(\bx)-d^2(\bz)}{2\sigma_{\bUpsilon}^2}} d\mu(\bx)}{\int_{\cK_\bz^{\epsilon/2}} e^{-\frac{h_\bz^2(\bx)-d^2(\bz)}{2\sigma_{\bUpsilon}^2}} d\mu(\bx)} \\
				&\leq \frac{e^{\frac{-\epsilon}{2\sigma_{\bUpsilon}^2}}\mu(\supp(\mu) \setminus \cU_{\bz}^\epsilon)}{e^{\frac{-\epsilon}{4\sigma_{\bUpsilon}^2}}\mu(\cK_{\bz}^{\epsilon/2})} 
				\leq \frac{e^{\frac{-\epsilon}{4\sigma_{\bUpsilon}^2}}\mu(\cX)}{\mu(\cK_\bz^{\epsilon/2})} .
			\end{align*}
			Passing to the limit as $\sigma_{\bUpsilon} \to 0$ we get
			\[
			\mu_{\bz}^{\sigma_{\bUpsilon}}(\supp(\mu) \setminus \cU_{\bz}^\epsilon) \to 0 .
			\]
			Compactness of $\cX$ implies that $\cP(\cX)$ is weak-$*$ compact by \cite[Theorem~15.11]{Aliprantis06}. Then, the family $(\mu_{\bz}^{\sigma_{\bUpsilon}})_{\sigma_{\bUpsilon} > 0}$ is relatively compact by Prokhorov's theorem. Thus, every weak-$*$ cluster point $\bar{\mu}_\bz$ is such that $\bar{\mu}_\bz(\supp(\mu) \setminus \cU^{\epsilon}_\bz) = 0$. Because $\epsilon>0$ is arbitrary and $\proj{\supp(\mu)}(\bz) = \bigcap_{\epsilon>0} \cU^\epsilon_\bz$, we get that $\supp(\bar{\mu}_\bz) \subset \proj{\supp(\mu)}(\bz)$. Indeed, for $\bx \notin \proj{\supp(\mu)}$, there must be $\epsilon >0$ with $\bx \in \pa{\cU_\bz^\epsilon}^{c}$. But since $\pa{\cU_\bz^\epsilon}^{c}$ is open and weak-*-convergence implies \begin{align*}
				\bar{\mu}_{\bz}\pa{\pa{\cU_\bz^\epsilon}^{c}} = \bar{\mu}_{\bz}\pa{\supp(\mu)\cap \pa{\cU_\bz^\epsilon}^{c}} + \bar{\mu}_{\bz}\pa{\supp(\mu)^{c}\cap \pa{\cU_\bz^\epsilon}^{c}} &= \bar{\mu}_{\bz}\pa{\supp(\mu)^{c}\cap \pa{\cU_\bz^\epsilon}^{c}} \\
				&\le \liminf_{\sigma_{\bUpsilon} \to 0} \mu_{\bz}^{\sigma_{\bUpsilon}}\pa{\supp(\mu)^{c}\cap \pa{\cU_\bz^\epsilon}^{c}} = 0,
			\end{align*} it must be $\bx \notin \supp(\bar{\mu}_{\bz})$.  As $\cX$ is compact, $\cW_p$ metrizes weak-$*$ convergence on $\cP(\cX)$ \cite[Theorem~6.9]{VillaniON09}. Therefore, subsequential weak-$*$ convergence is equivalent to that in $\cW_p$.\\
			
			\noindent\ref{prop:limsigma_claim2} For any subsequence $(\mu_{\bz}^{\sigma_{\bUpsilon,k}})_{k \in \bbN}$ (that we do not relabel) with weak-$*$ cluster point $\bar{\mu}_\bz$, we have from \ref{prop:limsigma_claim1},
			\[
			\widehat{D}^{\sigma_{\bUpsilon,k}}_{\mu}(\bz) = \int_{\cX} \bx d\mu_{\bz}^{\sigma_{\bUpsilon,k}}(\bx) \underset{k \to \infty}{\longrightarrow} \int_{\cX} \bx d\bar{\mu}_{\bz}(\bx) \in \conv{\proj{\supp(\mu)}(\bz)} = \widehat{D}^{0}(\bz) .
			\] 
			As this is true for any subsequence $(\mu_{\bz}^{\sigma_{\bUpsilon,k}})_{k \in \bbN}$, we get the limit claim. \\			
			Let us now prove that $\widehat{D}^{0}$ is maximal monotone. Let $(\bz,\bEta) \in \graph(\widehat{D}^{0})$ and $(\bz',\bEta') \in \graph(\widehat{D}^{0})$. By Carath\'eodory's theorem, $\bEta$ can be expressed as the convex combination
			\[
			\bEta = \sum_{i=1}^{r+1} \alpha_i \bEta_i \qwhereq \bEta_i \in \proj{\supp(\mu)}(\bz) \qandq \sum_{i=1}^{r+1}\alpha_i = 1, \alpha_i \geq 0 ~~ \forall i \in [r+1] .
			\]
			Similarly,
			\[
			\bEta' = \sum_{i=1}^{r+1} \beta_i \bEta_i' \qwhereq \bEta_i' \in \proj{\supp(\mu)}(\bz') \qandq \sum_{i=1}^{r+1}\beta_i = 1, \beta_i \geq 0 ~~ \forall i \in [r+1] .
			\]
			We then have
			\[
			\dotp{\bEta'-\bEta}{\bz'-\bz} =\dotp{\pa{\sum\limits_{j =1}^{r+1}\beta_j}\sum\limits_{i = 1}^{r+1}\alpha_i\bEta_i- \pa{\sum\limits_{i =1}^{r+1}\alpha_i}\sum\limits_{j = 1}^{r+1}\beta_j\bEta_j'}{\bz'-\bz} =\sum_{i,j=1}^{r+1} \alpha_i\beta_j \dotp{\bEta_j'-\bEta_i}{\bz'-\bz} .
			\]
			We know that $\proj{\supp(\mu)}$ is a monotone operator; see e.g., \cite[Example~20.12]{CombettesBauschke17}, whence monotonicity of $\widehat{D}^{0}$ follows. \\			
			We now turn to maximality. By Minty's theorem \cite[Theorem~21.1]{CombettesBauschke17}, $\widehat{D}^{0}$ is maximal monotone if and only if $I_r+\widehat{D}^{0}$ is surjective. That is, for every $\bEta \in V$, there exists $(\bz,\bb) \in \graph(\widehat{D}^{0})$ such that $\bEta = \bz + \bb$, or equivalently, for every $\bEta \in V$,
			\[
			\exists \bb \qstq \bb \in \widehat{D}^{0}(\bEta - \bb) .
			\]
			The latter statement amounts to proving the existence of a fixed point of the set-valued operator 
			\[
			T \eqdef \widehat{D}^{0}(\bEta - \cdot): \conv{\supp(\mu)} \to 2^{\conv{\supp(\mu)}} .
			\]
			This holds true thanks to Kakutani's fixed point theorem; see e.g., \cite[Corollary~1.12.1]{AubinCellina84}. Indeed, $T$ has non-empty convex compact images. In view of \cite[Corollary~1.1.1]{AubinCellina84}, it remains to show that $\graph(T)$ is closed, that is, for any sequence $(\bz_k,\bEta_k) \in \graph(T)$ such that $(\bz_k,\bEta_k) \to (\bar{\bz},\bar{\bEta})$, then $(\bar{\bz},\bar{\bEta}) \in \graph(T)$. By Carath\'eodory's theorem, for each $\bEta_k$, we can write
			\[
			\bEta_k = \sum_{i=1}^{r+1} \alpha_{i,k} \bEta_{k,i} \qwhereq \bEta_{k,i} \in \proj{\supp(\mu)}(\bEta-\bz_k) \qandq \sum_{i=1}^{r+1}\alpha_{i,k} = 1, \alpha_{i,k} \geq 0 ~~ \forall i \in [r+1] .
			\]
			By compactness of the standard simplex and the values of $\proj{\supp(\mu)}$, we can pass to a subsequence so that $\alpha_{i,k} \to \bar{\alpha}_i$ and $\bEta_{k,i} \to \bar{\bEta}_i$, where $\bar{\alpha}$ belongs to the standard simplex. Since $\proj{\supp(\mu)}$ has a closed graph by \cite[Example~1.20]{RockafellarWets}, we have $\bar{\bEta}_i \in \proj{\supp(\mu)}(\bEta-\bar{\bz})$. In turn,
			\[
			\bar{\bEta} = \lim_{k \to \infty} \bEta_k = \sum_{i=1}^{r+1} \bar{\alpha}_i \bar{\bEta}_i \in \conv{\proj{\supp(\mu)}(\bEta-\bar{\bz})} = T(\bar{\bz}) .
			\]
			This concludes the proof.
		\end{proof}
	\end{proposition}

	\section{Proofs of the main results}\label{sec:proofs}
	
	\subsection{Proof of Theorem~\ref{thm:pointwiserecovbndpnpfbs}}
	
	In order to show \eqref{eq:pointwise_recovery_1}, we need the following lemma.
	
	\begin{lemma}\label{lem:bound_MMSE_denoising_error}
		Assume that \eqref{assum:Xcompact}--\eqref{assum:Xnoise} hold, and that $\mu_{\bX}$ fulfills \eqref{assum:XV}. Then there exists $\epsilon >0$ such that  
		\begin{equation}
			\norm{\widehat{D}_{\mu_\bX^N}\pa{\bx} - \bx} \leq \sqrt{\frac{\lambda_{\max}(\widetilde{\Sigma}_{\bUpsilon})}{\lambda_{\min}(\widetilde{\Sigma}_{\bUpsilon})}}\dist(\bx,\cD_N^\bX) +  2M\pa{N-1} e^{-\frac{\epsilon}{2\lambda_{\max}\pa{\widetilde{\Sigma}_{\bUpsilon}}}} .
		\end{equation}
		\begin{proof}
			Let $\cI_{\bx} \eqdef \Argmin_{i \in [N]} \norm{\bx - \bx^i}_{\widetilde{\Sigma}_{\bUpsilon}^+}$ and $d_\star \eqdef \min_{i \in [N]} \norm{\bx - \bx^i}_{\widetilde{\Sigma}_{\bUpsilon}^+}$. For all $i \in \cI_{\bx}$, it holds $\norm{\bx-\bx^i}_{\widetilde{\Sigma}_{\bUpsilon}^+} = d_\star$. From \eqref{eq:Ndmmsegauss} and \eqref{eq:Dmmsemu}, and since $\bx \in V$ and $\cD_N^\bX \subset V$ by \eqref{assum:XV}, we have
			\begin{align*}
				\norm{\widehat{D}_{\mu_\bX^N}\pa{\bx} - \bx} 
				&= \norm{\frac{\sum_{i=1}^N \bx^i e^{-\frac{\norm{\bx-\bx^i}_{\widetilde{\Sigma}_{\bUpsilon}^+}^2}{2}}}{\sum_{i=1}^Ne^{-\frac{\norm{\bx-\bx^i}_{\widetilde{\Sigma}_{\bUpsilon}^+}^2}{2}}} - \bx} 
				\leq \frac{\sum_{i=1}^N \norm{\bx-\bx^i} e^{-\frac{\norm{\bx-\bx^i}_{\widetilde{\Sigma}_{\bUpsilon}^+}^2}{2}}}{\sum_{i=1}^N e^{-\frac{\norm{\bx-\bx^i}_{\widetilde{\Sigma}_{\bUpsilon}^+}^2}{2}}} \\
				&\leq \frac{\sum_{i \in \cI_{\bx}}\norm{\bx-\bx^i} + \sum_{i \not\in\cI_{\bx}} \norm{\bx-\bx^i} e^{-\frac{\pa{\norm{\bx-\bx^i}_{\widetilde{\Sigma}_{\bUpsilon}^+}^2 - d_\star^2}}{2}}}{|\cI_{\bx}| + \sum_{i \not\in\cI_{\bx}} e^{-\frac{\pa{\norm{\bx-\bx^i}_{\widetilde{\Sigma}_{\bUpsilon}^+}^2 - d_\star^2}}{2}}} \\
				&\leq \frac{1}{|\cI_{\bx}|}\sum_{i \in \cI_{\bx}}\norm{\bx-\bx^i} + \frac{1}{|\cI_{\bx}|}\sum_{i \not\in\cI_{\bx}} \norm{\bx-\bx^i} e^{-\frac{\pa{\norm{\bx-\bx^i}_{\widetilde{\Sigma}_{\bUpsilon}^+}^2 - d_\star^2}}{2}}.
			\end{align*}
			Let $i_\star \in \Argmin_{i \in [N]}\norm{\bx-\bx^i}$, i.e. $\norm{\bx-\bx^{i_\star}} = \dist(\bx,\cD_N^\bX)$. Since for all $i\in\cI_x$, $\norm{\bx-\bx^i}_{\widetilde{\Sigma}_{\bUpsilon}^+} = d_\star \leq \norm{\bx-\bx^{i_\star}}_{\widetilde{\Sigma}_{\bUpsilon}^+}$, we have		
			\begin{align*}
                \norm{\bx-\bx^i}
                \leq \sqrt{\lambda_{\max}(\tilde{\Sigma}_{\bUpsilon})}\norm{\bx-\bx^i}_{\widetilde{\Sigma}_{\bUpsilon}^+}
                \leq \sqrt{\lambda_{\max}(\tilde{\Sigma}_{\bUpsilon})}\norm{\bx-\bx^{i_\star}}_{\widetilde{\Sigma}_{\bUpsilon}^+}
                \leq \sqrt{\frac{\lambda_{\max}(\tilde{\Sigma}_{\bUpsilon})}{\lambda_{\min}(\tilde{\Sigma}_{\bUpsilon})}}\dist(\bx,\cD_N^\bX) .
            \end{align*}
            In turn,
            \[
			\norm{\widehat{D}_{\mu_\bX^N}\pa{\bx} - \bx} 
			\leq \sqrt{\frac{\lambda_{\max}(\tilde{\Sigma}_{\bUpsilon})}{\lambda_{\min}(\tilde{\Sigma}_{\bUpsilon})}}\dist(\bx,\cD_N^\bX) + \frac{1}{|\cI_{\bx}|}\sum_{i \not\in\cI_{\bx}} \norm{\bx-\bx^i} e^{-\frac{\pa{\norm{\bx-\bx^i}_{\widetilde{\Sigma}_{\bUpsilon}^+}^2 - d_\star^2}}{2}}.
			\] 
			If $|\cI_{\bx}|=N$, the claim follows. Otherwise, there exists $\epsilon > 0$ such that $\norm{\bx-\bx^i}_{\widetilde{\Sigma}_{\bUpsilon}^+}^2 - d_\star^2 \geq \epsilon$ for all $i \not\in \cI_{\bx}$, and therefore, using also \eqref{assum:Xcompact},
			\begin{align*}
			\norm{\widehat{D}_{\mu_\bX^N}\pa{\bx} - \bx} 
			&\leq \sqrt{\frac{\lambda_{\max}(\tilde{\Sigma}_{\bUpsilon})}{\lambda_{\min}(\tilde{\Sigma}_{\bUpsilon})}}\dist(\bx,\cD_N^\bX) +  \diam(\cX)\pa{N/|\cI_{\bx}|-1} e^{-\frac{\epsilon}{2\lambda_{\max}\pa{\widetilde{\Sigma}_{\bUpsilon}}}} \\
			&\leq \sqrt{\frac{\lambda_{\max}(\tilde{\Sigma}_{\bUpsilon})}{\lambda_{\min}(\tilde{\Sigma}_{\bUpsilon})}}\dist(\bx,\cD_N^\bX) +  2M\pa{N-1} e^{-\frac{\epsilon}{2\lambda_{\max}\pa{\widetilde{\Sigma}_{\bUpsilon}}}} .
			\end{align*}
		\end{proof}
	\end{lemma}
	\begin{remark}\label{rem:epsN}
		Note that $\epsilon$ in Lemma~\ref{lem:bound_MMSE_denoising_error} depends on $N$. This dependence can be removed by choosing $\epsilon >0$ and repeating the arguments of the proof with 
		\begin{equation*}
			\cI_{\bx,\epsilon} = \enscond{i\in [N]}{\norm{\bx-\bx^i}_{\widetilde{\Sigma}_{\bUpsilon}^+}^2 \le d_\star^2+\epsilon}.
		\end{equation*}
		instead of $\cI_{\bx}$, to obtain the estimate 
		\begin{equation*}
			\norm{\widehat{D}_{\mu_\bX^N}\pa{\bx} - \bx} \leq \sqrt{\frac{\lambda_{\max}(\tilde{\Sigma}_{\bUpsilon})}{\lambda_{\min}(\tilde{\Sigma}_{\bUpsilon})}}\dist(\bx,\cD_N^\bX) +  2M\pa{N-1} e^{-\frac{\epsilon}{2\lambda_{\max}\pa{\widetilde{\Sigma}_{\bUpsilon}}}} + \sqrt{\lambda_{\max}(\tilde{\Sigma}_{\bUpsilon})\epsilon} .
		\end{equation*}
	\end{remark}
	\,\\
	We are now in the position to prove Theorem~\ref{thm:pointwiserecovbndpnpfbs}.
	\begin{proof}[Proof of \eqref{eq:pointwise_recovery_1}]
		Note that due to \eqref{assum:BA}, conditon \eqref{assum:restinj} is equivalent to
		\[
		\lambda_{\min}\pa{BA;T_{\conv{\supp(\mu_\bX)}}(\bx)} > 0 .
		\]
		Therefore, using $\bx_k\in \conv{\supp(\mu)}$, we obtain for $\gamma >0$ that 
		\begin{align*}
			\norm{(I_n - \gamma B A)(\bx_k - \bx)}^2 
			&= \norm{\bx_k - \bx}^2 - 2 \gamma\dotp{BA(\bx_k - \bx)}{\bx_k - \bx} + \gamma^2\norm{BA(\bx_k - \bx)}^2 \\
			&\leq \rho(\gamma)^2\norm{\bx_k - \bx}^2 ,
		\end{align*}
		where $\rho(\gamma)^2 = 1-2\gamma\lambda_{\min}\pa{BA;T_{\conv{\supp(\mu_\bX)}}(\bx)}+\gamma^2\lambda_{\max}(BA)^2$. By \eqref{eq:condnum} it holds $\rho(\gamma)\ge 0$. In particular, $\rho(\gamma)$ is minimal with value $\rho(\gamma) = \sqrt{1 - \kappa\pa{BA}^{-2}} <1$ for $\gamma$ chosen according to \eqref{eq:optgamma}. Since $q = \frac{M^2}{\lambda_{\min}\pa{\widetilde{\Sigma}_{\bUpsilon}}} \sqrt{1 - \kappa\pa{BA}^{-2}} <1$ by \eqref{eq:cond_nonexpansive} or \eqref{eq:cond_welldefined}, we obtain from \eqref{eq:bndpnpfbs_1} that
		\begin{align}
			\norm{\bx_{k} - \bx}
			&\leq q\norm{\bx_{k-1} - \bx} + \lip{\widehat{D}_{\mu_{\bX}^N}}\gamma\norm{B}\norm{\be} + \norm{\widehat{D}_{\mu_\bX^N}\pa{\bx} - \bx} \nonumber\\
			&\leq {q}^{k+1}\norm{\bx_0 - \bx} + \frac{1-{q}^{k}}{1-{q}}\pa{\lip{\widehat{D}_{\mu_{\bX}^N}}\gamma \norm{B}\norm{\be} + \norm{\widehat{D}_{\mu_\bX^N}\pa{\bx} - \bx}} . \label{eq:bndpnpfbsnonexpansive_1}
		\end{align}
		Estimate \eqref{eq:pointwise_recovery_1} now follows from Lemma~\ref{lem:bound_MMSE_denoising_error}.
	\end{proof}
	
	Estimate \eqref{eq:pointwise_recovery_2} is a consequence of \eqref{eq:pointwise_recovery_1} and the following lemma applied with $\tau=2$. 
	
	\begin{lemma}\label{lem:bnddistxD_N}
		Assume that \eqref{assum:Xcompact} holds and that $\mu_{\bX}$ satisfies \eqref{assum:XV} and has a density $g$ with respect to the volume measure on $\bbR^r$ such that $g$ is bounded away from zero. Then, for any $\tau > 0$ there exists $N(\tau) \in \N$ such that for all $N \geq N(\tau)$, 
		\begin{equation*}
			\max_{\bz \in \supp(\mu_{\bX})} \dist(\bz,\cD_N^\bX) \leq K_1(1+\tau)^{1/r}\pa{\frac{\log(N)}{N}}^{1/r} ,
		\end{equation*}
		holds with probability at least $1 - K_2 N^{-\tau}$ on the sampling of $\cD_N^\bX$, where $K_1 > 0$ and $K_2 > 0$ are two constants that depend only on $r$ and $M$. 
		\begin{proof}
			Denote for short $\cS \eqdef \supp(\mu_{\bX}) \subset V$ and identify $V$ with $\bbR^r$.
			
			Let $S_\delta = \ens{\bz_1,\bz_2,\ldots,\bz_{N(\cS,\delta)}}$ be a $\delta$-net of $\cS$ in the Euclidean distance on $\bbR^r$, i.e., $\cS \subseteq \bigcup_{\bz \in S_\delta} \clBall_\delta(\bz)$. We then have
			\begin{align*}
				\max_{\bz \in \cS} \dist(\bz,\cD_N^\bX) \leq \max_{j \in [N(\cS,\delta)]}\max_{\bz \in \clBall_\delta(\bz_j)} \dist(\bz,\cD_N^\bX) .
			\end{align*}
			For each $j \in [N(\cS,\delta)]$, let $\bQ_j$ be the number of random variables in $\cD_N^\bX$ falling into $\clBall_{\delta}(\bz_j)$. $\bQ_j$ is a Binomial random variable with parameters $(N,p_j)$, where $p_j = \mu_{\bX}(\clBall_\delta(\bz_j)) \geq c \vol(\clBall_\delta(0)) = c \delta^r\vol(\clBall(0))$, where $c = \inf_{\cS} g > 0$ by assumption, and we used the shorthand notation $\clBall(0)$ for the unit ball centered at zero. 
            Thus, using the union bound, we get
			\begin{align*}
				&\Pr\pa{\max_{\bz \in \cS} {\dist(\bz,\cD_N^\bX)} > 2\delta} 
				\leq \Pr\pa{\max_{j \in [N(\cS,\delta)]}{\max_{\bz \in \clBall_\delta(x_j)} {\dist(\bz,\cD_N^\bX)}} > 2\delta} \\
				&\leq \sum_{j \in [N(\cS,\delta)]}\Pr\pa{\max_{\bz \in \clBall_\delta(x_j)} {\dist(\bz,\cD_N^\bX)} > 2\delta} \\
				&\leq \sum_{j \in [N(\cS,\delta)]}\Pr\pa{\bQ_j = 0} = \sum_{j \in [N(\cS,\delta)]}(1-p_j)^N \leq N(\cS,\delta) \pa{1-c \delta^r\vol(\Ball(0))}^N .
			\end{align*}
			Since $\cS$ is compact, we know that $\cS \subseteq \clBall_M(0)$. It then follows from \cite[Lemma~4.10]{Pisier89} that
			\[
			N(\cS,\delta) = N(\cS/M,\delta/M) \leq \pa{1+\frac{2M}{\delta}}^r .
			\]
			We therefore arrive at the bound
			\begin{align*}
				\Pr\pa{\max_{x \in \cS} \dist(\bz,\cD_N^\bX) > 2\delta} 
				\leq \pa{1+\frac{2M}{\delta}}^r \pa{1-c \delta^r \vol(\Ball(0))}^N
				\leq \frac{(\delta+2M)^r}{\delta^r}e^{-Nc\delta^r\vol(\Ball(0)}.
			\end{align*}
			Take $\delta^r = \frac{(1+\tau)}{c\vol(\Ball(0))}\frac{\log(N)}{N}$, for any $\tau > 0$. Thus, for $N$ large enough, one has $\delta \leq M$, and in turn the above bound becomes
			\begin{align*}
				\Pr\pa{\max_{x \in \cS} \dist(\bz,\cD_N^\bX) > 2\delta} 
				&\leq c (3M)^r \vol(\Ball(0))(1+\tau)^{-1} \frac{N^{-\tau}}{\log(N)} \\
				&\leq c (3M)^r \vol(\Ball(0)) N^{-\tau} .
			\end{align*}
			By the Stirling formula, we have
			\[
			\vol(\Ball(0))=\frac{2\pi^{r/2}}{r\Gamma(r/2)} = \frac{1}{\sqrt{\pi r}} \pa{\frac{2\pi e}{r}}^{r/2} e^{\theta(r/2)/(6r)}
			\] 
			with $\theta(r/2) \in [0,1]$. Plugging this into the expression of $\delta$ concludes the proof.
		\end{proof}
	\end{lemma}
	
	\subsection{Proof of Theorem~\ref{thm:Wrecovbndpnpfbs}}
	\begin{proof}[Proof of Lemma~\ref{lem:uniformization_restricted_injectivity}]
			Denote for short $\cU_\bx \eqdef T_{\conv{\cS}}(\bx) \cap \sph^{n-1}$ for any $\bx \in \cS \eqdef {\supp(\mu_\bX)}$, and note that it holds 
			\begin{equation*}
				\bigcup\limits_{\bx \in {\cS}}\cU_{\bx} = \pa{\bigcup\limits_{\bx \in {\cS}} T_{\conv{\cS}}(\bx)} \cap \sph^{n-1}.
			\end{equation*}
			By definition, we have for any $\bx \in \cS$,
			\[
			T_{\conv{\cS}}(\bx) \subset \Par(\conv{\cS}) ,
			\]
			and in turn
			\[
			\bigcup\limits_{\bx \in {\cS}} T_{\conv{\cS}}(\bx) \subset \Par(\conv{\cS}) .
			\]
			Moreover, for any $\bx \in \ri(\conv{\cS})$, we know that (see e.g., \cite[Theorem~6.45]{CombettesBauschke17})
			\[
			T_{\conv{\cS}}(\bx) = \Par(\conv{\cS}) .
			\]
			Altogether, since by \eqref{assum:restinj_unif} it holds that $\ri(\conv{\cS})\cap \cS$ is non-empty we deduce that
			\[
			\bigcup\limits_{\bx \in {\cS}} T_{\conv{\cS}}(\bx) = \Par(\conv{\cS}) .
			\]
			We conclude that $\bigcup\limits_{\bx \in {\cS}}\cU_{\bx}=\Par(\conv{\cS}) \cap \sph^{n-1}$ is compact since it is the unit sphere of the finite-dimensional vector space $\Par(\conv{\cS})$. We thus obtain 
			\begin{equation*}
				\widetilde{\lambda}_{\min}(BA) = \inf_{\bz \in \bigcup\limits_{\bx \in \cS} \cU_\bx} \norm{BA \bz} = \min\limits_{\bz \in \Par(\conv{\cS})\cap \sph^{n-1}}\norm{BA \bz} >0,
			\end{equation*}	
			where the last inequality follows from assumption \eqref{assum:restinj_unif}.
	\end{proof}
	
	\begin{proof}[Proof of Theorem~\ref{thm:Wrecovbndpnpfbs}]
		Let $\bE$ and $\bUpsilon$ be such that $\widetilde{\gamma} B\bE$ and $\bUpsilon$ are jointly distributed according to their optimal coupling, and similarly for $\bX$ and $\bX^N$. Let $\bY = A\bX+\bE$. We then have, conditioned on the samples $\cD_N^\bX$, that
		\[
		\cW_2(\widehat{\mu}_k^N,\mu_{\bX}) \leq \sqrt{\Expect{}{\norm{\bX_k - \bX}^2\mid\cD_N^\bX}} ,
		\]
		where $\bX_k$ is the iterate defined as in \eqref{eq:xkx0y} with $\bX_0$ and $\bY$ as above.  Using \eqref{eq:bndpnpfbs_2} and Minkowski's inequality, we obtain
		\begin{align*}
			&\sqrt{\Expect{}{\norm{\bX_k - \bX}^2\mid\cD_N^\bX}}
			\leq \widetilde{q}\sqrt{\Expect{}{\norm{\bX_{k-1} - \bX}^2\mid\cD_N^\bX}} + \frac{M^2}{\lambda_{\min}(\widetilde{\Sigma}_{\bUpsilon})}\sqrt{\Expect{}{\norm{\widetilde{\gamma} B\bE - \bUpsilon}^2}} \\
			&+ \sqrt{\Expect{}{\norm{\widehat{D}_{\mu_\bX^N}\pa{\bX^N + \bUpsilon} - \bX^N}^2\mid\cD_N^\bX}} 
			+ \pa{1+\frac{M^2}{\lambda_{\min}(\widetilde{\Sigma}_{\bUpsilon})}}\sqrt{\Expect{}{\norm{\bX^N - \bX}^2\mid\cD_N^\bX}} \\
			&= \widetilde{q}\sqrt{\Expect{}{\norm{\bX_{k-1} - \bX}^2\mid\cD_N^\bX}} + \frac{M^2}{\lambda_{\min}(\widetilde{\Sigma}_{\bUpsilon})}\cW_2(\mu_{\widetilde{\gamma} B\bE},\mu_{\bUpsilon}) \\
			&+ \sqrt{\Expect{}{\norm{\widehat{D}_{\mu_\bX^N}\pa{\bX^N + \bUpsilon} - \bX^N}^2\mid\cD_N^\bX}} 
			+ \pa{1+\frac{M^2}{\lambda_{\min}(\widetilde{\Sigma}_{\bUpsilon})}}\cW_2(\mu_{\bX},\mu_{\bX}^N) .
		\end{align*}
		Iterating, we obtain
		\begin{align}\label{eq:estimate_proof_Wasserstein}
		\cW_2(\widehat{\mu}_k^N,\mu_{\bX}) 
		&\leq \widetilde{q}^{k+1}\sqrt{\Expect{}{\norm{\bX_{0} - \bX}^2\mid\cD_N^\bX}} + \frac{1-\widetilde{q}^{k}}{1-\widetilde{q}}\Bigg[\frac{M^2}{\lambda_{\min}(\widetilde{\Sigma}_{\bUpsilon})}\cW_2\pa{(\widetilde{\gamma} B)_{\#}\mu_{\bE},\mu_{\bUpsilon}} \nonumber\\
			&+ \sqrt{\Expect{}{\norm{\widehat{D}_{\mu_\bX^N}\pa{\bX^N + \bUpsilon} - \bX^N}^2\mid\cD_N^\bX}} 
			+ \pa{1+\frac{M^2}{\lambda_{\min}(\widetilde{\Sigma}_{\bUpsilon})}}\cW_2(\mu_{\bX},\mu_{\bX}^N)\Bigg].
		\end{align}
		Now taking the expectation over the sampling of $\cD_N^\bX$ and using Jensen's inequality, we obtain
		\begin{align}\label{eq:estimate_proof_expect_Wasserstein}
		\Expect{\cD_N^\bX}{W_2(\widehat{\mu}_k^N,\mu_{\bX})}
		&\leq \widetilde{q}^{k+1}\sqrt{\Expect{}{\norm{\bX_{0} - \bX}^2}} + \frac{1-\widetilde{q}^{k}}{1-\widetilde{q}}\Bigg[\frac{M^2}{\lambda_{\min}(\widetilde{\Sigma}_{\bUpsilon})}\cW_2((\widetilde{\gamma} B)_{\#}\mu_{\bE},\mu_{\bUpsilon}) \nonumber\\
			&+ \sqrt{\Expect{}{\norm{\widehat{D}_{\mu_\bX^N}\pa{\bX^N + \bUpsilon} - \bX^N}^2}} 
			+ \pa{1+\frac{M^2}{\lambda_{\min}(\widetilde{\Sigma}_{\bUpsilon})}}\Expect{\cD_N^\bX}{\cW_2(\mu_{\bX},\mu_{\bX}^N)}\Bigg].
		\end{align}
		We bound the last three summands in \eqref{eq:estimate_proof_expect_Wasserstein} individually. Since $\bE$ and $\bUpsilon$ are zero-mean Gaussian, the first term is given by the Bures metric valid even in the degenerate case thanks to \cite[Theorom~2.1 and Theorom~2.4]{Gelbrich90}, 
		\begin{align*}
			\cW_2((\widetilde{\gamma} B)_{\#}\mu_{\bE},\mu_{\bUpsilon}) = \pa{\widetilde{\gamma}^{^2}\tr{B \Sigma_\bE B^\top} + \tr{\widetilde{\Sigma}_{\bUpsilon}} 
				- 2\widetilde{\gamma}\tr{\widetilde{\Sigma}_{\bUpsilon}^{1/2}\Pi^\top B \Sigma_\bE B^\top\Pi\widetilde{\Sigma}_{\bUpsilon}^{1/2}}^{1/2}}^{1/2}.
		\end{align*}
		For the second term,  we argue as in the proof of Theorem~\ref{thm:lipDmmseW2}\ref{thm:lipDmmseW2_claim2}, using optimality of the MMSE denoiser $\widehat{D}_{\mu_\bX^N}$ (see \eqref{eq:mmseoptbnd}) to get that
	\begin{align*}
	\Expect{}{\norm{\widehat{D}_{\mu_\bX^N}\pa{\bX^N + \bUpsilon} - \bX^N}^2} 
	&= \Expect{\cD_N^\bX}{\norm{\widehat{D}_{\mu_\bX^N}\pa{\bX^N + \bUpsilon} - \bX^N}^2\mid \cD_N^\bX} \\
	&= \Expect{\cD_N^\bX}{\cR(\widehat{D}_{\mu_{\bX}^N})} \leq \Expect{\cD_N^\bX}{\norm{\bUpsilon}^2} =  \tr{\widetilde{\Sigma}_{\bUpsilon}}
			\leq \lambda_{\max}(\widetilde{\Sigma}_{\bUpsilon}) r .
		\end{align*}
		For the sampling error term $\Expect{\cD_N^\bX}{\cW_2(\mu_{\bX},\mu_{\bX}^N)}$, we argue as in Corollary~\ref{cor:lipDmmseempiricalW2} to get that it scales as $O(\diam(\cX)N^{-1/s})$. Plugging these  estimates into \eqref{eq:estimate_proof_expect_Wasserstein}, we get the claimed bound.
%
	\end{proof}
	
	\subsection{Proof of Proposition~\ref{prop:high_prob_Wasserstein}}
	From the proof of Theorem~\ref{thm:Wrecovbndpnpfbs}, we only need to show that $\cW_2(\mu_{\bX},\mu_{\bX}^N)$  concentrates well around its expectation. For this, we apply \cite[Proposition~20]{Weed19} to get that with probability larger than $1-\exp(-2Nt^2)$,
	\[
	\cW_2^2(\mu_{\bX},\mu_{\bX}^N) \leq \Expect{\cD_N^\bX}{\cW_2^2(\mu_{\bX},\mu_{\bX}^N)} + \diam(\cX)^2 t .
	\]
	The term $\Expect{\cD_N^\bX}{\cW_2^2(\mu_{\bX},\mu_{\bX}^N)}$ needs some care because of the square inside the expectation (Jensen's inequality is in the opposite direction). For this, we use the reverse inequality of \eqref{eq:Wporder} which holds thanks to boundedness of $\cX$ (see \ref{assum:Xcompact}). 
	
\begin{lemma}
Let $\mu,\nu$ be probability measures supported on $\cX$ which satisfies \eqref{assum:Xcompact}. Then, for $1 \leq p \leq q< +\infty$,
\[
\cW_q(\mu,\nu) \leq \diam(\cX)^{1-p/q}\cW_p(\mu,\nu)^{p/q}.
\]
\end{lemma}
The proof uses basic properties of the Wasserstein distance and we include it for self-containedness.

\begin{proof}
Let $\pi_p\in\Pi(\mu,\nu)$ (resp. $\pi_q\in\Pi(\mu,\nu)$) the optimal coupling for $\cW_p(\mu,\nu)$ (resp. $\cW_q(\mu,\nu)$). Both exist thanks to compactness. Since $\mu$ and $\nu$ are supported on $\cX$, we have
\[
\norm{\bx-\bz}^q = \norm{\bx-\bz}^{q-p}\norm{\bx-\bz}^p \leq \diam(\cX)^{q-p}\norm{\bx-\bz}^p .
\]
Therefore
\begin{align*}
W_q^q(\mu,\nu)
&= \int_{\cX \times \cX} \norm{\bx-\bz}^q d\pi_q(\bx,\bz) \\
&\leq \int_{\cX \times \cX} \norm{\bx-\bz}^q d\pi_p(\bx,\bz) \\
&\leq \diam(\cX)^{q-p} \int_{\cX \times \cX} \norm{\bx-\bz}^p d\pi_p(\bx,\bz) = \diam(\cX)^{q-p} \cW_p^p(\mu,\nu) .
\end{align*}
\end{proof}

It then follows that with the same probability as above,
\[
\cW_2^2(\mu_{\bX},\mu_{\bX}^N) \leq \diam(\cX)\Expect{\cD_N^\bX}{\cW_1(\mu_{\bX},\mu_{\bX}^N)} + \diam(\cX)^2 t .
\]
We now use \cite[Corollary~1.2]{Boissard14} to see that
\[
\Expect{\cD_N^\bX}{\cW_1(\mu_{\bX},\mu_{\bX}^N)} \lesssim \diam(\cX) N^{-1/s'} 
\]
for $s' > d_{\mathrm{M}}(\mu_\bX) \geq 2$. Therefore, with the same probability as above,
\[
\cW_2(\mu_{\bX},\mu_{\bX}^N) \lesssim \diam(\cX)N^{-1/(2s')} + \diam(\cX) \sqrt{t} . 
\]
Setting $s=2s'$ and $t=N^{-1/s'}$, and plugging this into \eqref{eq:estimate_proof_Wasserstein}, we conclude.
	
	\subsection{Proof of Lemma~\ref{lem:NN_existence_uniform}}
			Since the identity is $1$-Lipschitz, the claim follows from \cite[Theorem~3.9]{Petersen23} and a scaling argument.
		
        Let  $\widetilde{\cS} \eqdef\diam(\cS)^{-1}\cS$ and choose $b\in \R^r$ such that $\widetilde{\cS}-\{b\}\subset [0,1]^r$ holds. Since the identity is $1$-Lipschitz, it follows from \cite[Theorem~3.9]{Petersen23} that there exists a neural network $D$ with $W(D) \lesssim rR^{d_{\cS}+1}\pa{\frac{\epsilon}{\diam(\cS)}}^{-d_{\cS}}$, and $L(D) = \lceil \log_2(d_{\cS}) \rceil + \lceil \log_2(r) \rceil + c$ for some universal constant $c\in \N$ such that \begin{equation*}
            \norm{D(\bx) - \bx} \leq \diam(\cS)^{-1}\epsilon 
        \end{equation*}
        for all  $\bx \in \widetilde{\cS}-\{b\}$. Moreover, according to \cite[Remark~3.12]{Petersen23}, there exists $\widetilde{p}\in \N$, which only depends on $\cS$, such that the weights of $D$ are bounded proportionally to $\pa{\epsilon \diam(\cS)^{-1}}^{-\widetilde{p}}$. To be precise, the upper bound of the weights can be chosen as the maximum of the optimal scaling order in \cite[Lemma~A.3]{Petersen2018}, and the inverse grid size that is necessary to approximate the identity with precision $\diam(\cS)^{-1}\epsilon $ using linear finite element functions. Since both of these quantitities are polynomial with respect to the inverse approximation error $\diam(\cS)\epsilon^{-1}$ and only dependent on $\cS$, such $\widetilde{p}$ must exist. We therefore obtain that the neural network $\widehat{D}_\epsilon = \diam(\cS)D$ approximates the identity on $\cS$ with precision $\epsilon$. In particular, $\widehat{D}_\epsilon$ has the same number of weights and layers as $D$. Since $\norm{b}\le \sqrt{r}\frac{M}{\diam(S)}$ and $\diam(\cS) \le 2M$, the weights of $\widehat{D}_\epsilon$ are bounded proportionally to $\sqrt{r}\pa{\frac{\epsilon}{M}}^{-p}$ with $p = \widetilde{p}+1$.
    
	\subsection{Proof of Theorem~\ref{thm:recovery_pointwise_NN}}
		Analogously to the proof of Theorem~\ref{thm:pointwiserecovbndpnpfbs}, we obtain \begin{equation*}
			\norm{\bx_{k} - \bx} \le {q}^{k+1}\norm{\bx_0 - \bx} + \frac{1-{q}^{k}}{1-{q}}\pa{L_{\widehat{D}_\epsilon^N}\gamma \norm{B}\norm{\be} + \norm{\widehat{D}_{\epsilon}^N\pa{\bx} - \bx}}.
		\end{equation*}
		To bound the last summand, we let $i\in\Argmin\limits_{j\in[N]}\norm{\bx^j-\bx}$ and make the estimate
		\begin{align}\label{eq:proofNNpointw1}
			\norm{\widehat{D}_\epsilon^N(\bx) -\bx} &\le \norm{\widehat{D}_\epsilon^N(\bx)-\widehat{D}_\epsilon^N(\bx^i+\bupsilon^i)}+ \norm{\widehat{D}_{\epsilon}^N(\bx^i+\bupsilon^i)-\bx^i} + \dist(\bx,\cD_N^\bX)\notag\\
			&\le \lip{\widehat{D}_\epsilon^N}\norm{\bx-(\bx^i+\bupsilon^i)} + \norm{\widehat{D}_{\epsilon}^N(\bx^i+\bupsilon^i)-\bx^i} + \dist(\bx,\cD_N^\bX)\notag\\
			&\le \lip{\widehat{D}_\epsilon^N} \norm{\bupsilon^i} + \sqrt{\sum\limits_{j = 1}^N \norm{\widehat{D}_\epsilon^N(\bx^j+\bupsilon^j)-\bx^j}^2} + \pa{1+\lip{\widehat{D}_\epsilon^N}}\dist(\bx,\cD_N^\bX).
		\end{align}
		By the optimality of $\widehat{D}_\epsilon^N$ in \eqref{eq:loss} and Lemma~\ref{lem:NN_existence_uniform} it holds \begin{align*}
			\sum\limits_{j = 1}^N \norm{\widehat{D}_\epsilon^N(\bx^j+\bupsilon^j)-\bx^j}^2 &\le \sum\limits_{j = 1}^N \norm{\widehat{D}_\epsilon(\bx^j+\bupsilon^j)-\bx^j}^2 \\
			&\le 2\sum\limits_{j = 1}^N \pa{\norm{\widehat{D}_\epsilon(\bx^j+\bupsilon^j)-\widehat{D}_\epsilon(\bx^j)}^2 + \norm{\widehat{D}_\epsilon(\bx^j)-\bx^j}^2} \\
			&\le 2\pa{{\lip{\widehat{D}_\epsilon}^2}\sum\limits_{j = 1}^N\norm{\bupsilon^j}^2 + N\epsilon^2}.
		\end{align*}
		Hence, estimate \eqref{eq:proofNNpointw1} implies \begin{equation*}
			\norm{\widehat{D}_\epsilon^N(\bx) -\bx} \le \pa{\lip{\widehat{D}_\epsilon^N} +\sqrt{2}\lip{\widehat{D}_\epsilon} }\pa{\sum\limits_{j = 1}^N\norm{\bupsilon^j}^2}^{\frac{1}{2}} + \sqrt{2N}\epsilon + \pa{1+\lip{\widehat{D}_\epsilon^N}}\dist(\bx,\cD_N^\bX) .
		\end{equation*}
		It remains to bound $f(\bupsilon^1,\ldots,\bupsilon^N)$, where $(\bupsilon^1,\ldots,\bupsilon^N) \mapsto f(\bupsilon^1,\ldots,\bupsilon^N) = \pa{\sum\limits_{j = 1}^N\norm{\bupsilon^j}^2}^{\frac{1}{2}}$. Since $\bUpsilon = {\widetilde{\Sigma}^{1/2}_{\bUpsilon}} \bZ$, for $\bZ \sim \cN(0,I_r)$, we have
		\[
		f(\bupsilon^1,\ldots,\bupsilon^N) \leq \sqrt{\lambda_{\max}(\widetilde{\Sigma}_{\bUpsilon})}f(\bz^1,\ldots,\bz^N) .
		\] 
		Observe that $f$ is $1$-Lipschitz, and since the $\bz^i$'s are independent zero-mean standard Gaussian vectors, we can use the concentration in Gauss space, see e.g., \cite{vershynin_high-dimensional_2018}, to get that
		\[
		f(\bupsilon^1,\ldots,\bupsilon^N) \leq \Expect{}{f(\bupsilon^1,\ldots,\bupsilon^N)} + t \leq \sqrt{Nr} + t
		\] 
		with probability at least $1-\exp(-t^2/2)$. Taking $t = \sqrt{4\log(N)}$, we get \eqref{eq:rcvbndNNptw1}.	
		Applying Lemma~\ref{lem:bnddistxD_N} with $t = N^{-2}$ and using a union bound implies \eqref{eq:rcvbndNNptw2}.

	\section{Conclusions and future work}
	We analyzed the forward-backward version of the PnP method using an empirical version of the optimal MMSE denoiser. In particular, we proved error bounds that loosen typical contractiveness assumptions on the denoiser for sufficiently well-conditioned problems. The crucial observation within our analysis, however, is that the denoiser should not be chosen independently of the forward model. Instead, it needs to be constructed such that it removes the transformed version of the observation noise, which occurs within the 'descent step' of the algorithm. We additionally show cocoercivity of the MMSE estimator, which implies firm nonexpansiveness up to scaling. In future work, we would like to study how this constraint can be rigorously and efficiently incorporated in the training of denoisers. Tightness of our bounds is also another interesting avenue for a future work.
	
	\subsection*{Acknowledgments}
	TW and JF were supported by the CAESAR AI4Science project from the University of Caen. All authors received support from the PHC HONG KONG-PROCORE Programme under number 52663RH.
	
	\appendices
	
	\section{General case of Lemma~\ref{lem:uniformization_restricted_injectivity}}\label{app:general_uniformization}
	Recall, see \cite{Luc_2008}, that the core outer limit of a sequence of sets $\pa{\cC_n}_{n\in \N}$ is defined via \begin{equation}
		{\limsup_{n\to\infty}}^\circ\, \cC_n = \left\{ x = \lim\limits_{n\to\infty} \lambda_n x_n: x_n \in \cC_n ,\ \lambda_n \uparrow \infty  \right\}.
	\end{equation}
	We show a generalized version of Lemma~\ref{lem:uniformization_restricted_injectivity}.
	
	\begin{lemma}
		Assume that \eqref{assum:Xcompact} holds and that $\ker(BA) \cap T_{\conv{\supp(\mu_\bX)}}(\bx) = \{0\}$ for all $\bx \in \supp(\mu_{\bX})$. Further assume that for each sequence $(\bx_n)_{n \in \N}$ it holds that \begin{equation}\label{eq:assum_coreouterlimit}
			{\limsup\limits_{n \to \infty}}^\circ\, T_{\conv{\supp(\mu_\bX)}}(\bx_n) \subset \bigcup\limits_{\bx\in \cS} T_{\conv{\supp(\mu_\bX)}}(\bx).
		\end{equation}
		Then $\widetilde{\lambda}_{\min}(BA) >0$.
		\begin{proof}
			We first show that $\bigcup\limits_{\bx\in \cS} T_{\conv{\supp(\mu_\bX)}}(\bx)$ is closed. Thus, let $(\bz_n)_{n\in\N}$ be a sequence in $T_{\conv{\supp(\mu_\bX)}}(\bx_n)$ with limit $\bz$. Hence, there exist $\lambda_n>0$ and $\by_n\in\conv{\supp(\mu_{\bX})} $ such that $\bz_n = \lambda_n(\by_n-\bx_n) $. We consider two cases. First, if $\lambda_n$ is unbounded, we can pass to a subsequence such that $\lambda_n$ increases to $\infty$. Hence $\bz \in {\limsup\limits_{n \to \infty}}^\circ\, T_{\conv{\supp(\mu_\bX)}}(\bx_n)$, and thus $\bz \in \bigcup\limits_{\bx\in \cS(\bx)} T_{\conv{\supp(\mu_\bX)}}(\bx)$. Otherwise, if $\lambda_n$ is bounded, we can pass to a subsequence and assume that $\bx_n$ converges to some $\bx \in \supp(\mu_{\bX})$ by compactness. Consequently, $\lambda_n(\by_n-\bx)$ converges to $\bz$, since \begin{equation*}
				\lambda_n(\by_n-\bx_n)-\lambda_n(\by_n-\bx) = \lambda_n(\bx-\bx_n) \overset{n\to \infty}{\to}0.
			\end{equation*}
			Thus, $\bz \in T_{\conv{\supp(\mu_\bX)}}(\bx)$. Therefore $\bigcup\limits_{\bx\in \cS} T_{\conv{\supp(\mu_\bX)}}(\bx)$ is closed and consequently \begin{equation*}
				\pa{\bigcup\limits_{\bx\in \cS} T_{\conv{\supp(\mu_\bX)}}(\bx)} \cap \sph^{n-1} =  \bigcup\limits_{\bx\in \cS}\pa{ T_{\conv{\supp(\mu_\bX)}}(\bx) \cap \sph^{n-1}}
			\end{equation*}
			is compact. We conclude \begin{equation*}
				\widetilde{\lambda}_{\min}(BA) = \inf\limits_{\bz \in \bigcup\limits_{\bx\in \cS}\pa{ T_{\conv{\supp(\mu_\bX)}}(\bx) \cap \sph^{n-1}}} \norm{BAz} =  \min\limits_{\bz \in \bigcup\limits_{\bx\in \cS}\pa{ T_{\conv{\supp(\mu_\bX)}}(\bx) \cap \sph^{n-1}}} \norm{BAz} >0.
			\end{equation*}
		\end{proof}
	\end{lemma}

	\section{General case of Lemma~\ref{prop:limsigma}}\label{append:ssymptotic}
	Let $\pa{\widetilde{\Sigma}_{\bUpsilon}^n}_{n\in\N}$ be a sequence of non-zero, symmetric, positive semidefinite matrices whose image is contained in $V$, such that  $\lim\limits_{n \to \infty}\widetilde{\Sigma}_{\bUpsilon}^n = 0$ holds and set \begin{equation*}
		\Lambda_n \eqdef \norm{\pa{\widetilde{\Sigma}_{\bUpsilon}^n}^+}^{-1} \pa{\widetilde{\Sigma}_{\bUpsilon}^n}^+.
	\end{equation*}
	Recall that for any symmetric positive-semidefinte matrix $\Lambda$, we set \begin{equation*}
		\norm{\bw}_\Lambda \eqdef \sqrt{\dotp{\Lambda \bw}{\bw}}. 
	\end{equation*}
	Further, fix $\bz \in V$ and let \begin{equation*}
		d_\Lambda(\bz){\eqdef}  \inf\limits_{\bx \in \supp(\mu)} \norm{\bx-\bz}_\Lambda,
	\end{equation*}
	and \begin{equation*}
		\projop^\Lambda_{\supp(\mu)}(\bz) \eqdef \enscond{\bx \in \supp(\mu)}{\norm{\bx-\bz}_\Lambda = d_\Lambda(\bz)}.
	\end{equation*}
	As in Proposition~\ref{prop:limsigma} define \begin{equation*}
		\mu_{\bz}^n(\cA) = \frac{\int_{\cA} e^{-\frac{1}{2}(\bx-\bz)^\top\pa{\widetilde{\Sigma}_{\bUpsilon}^n}^+(\bx-\bz)}d\mu(\bx)}{\int_{\cX} e^{-\frac{1}{2}(\bx-\bz)^\top\pa{\widetilde{\Sigma}_{\bUpsilon}^n}^+(\bx-\bz)}d\mu(\bx)}
	\end{equation*}
	for $\bz \in V$. 
	\begin{proposition}
		Assume that \eqref{assum:Xcompact} holds and that $\mu$ satisfies \eqref{assum:XV}.
		\begin{enumerate}[label=(\roman*)]
			\item 			$(\mu_{\bz}^n)_{n \in \N}$ subsequentially converges with respect to the Wasserstein distance $\cW_p$ and each cluster point $\bar{\mu}_{\bz}$ satisfies
			\[
			\supp(\bar{\mu}_{\bz}) \subset \projop^\Lambda_{\supp(\mu)}(\bz).
			\]
			\item 			For all $\bz \in V$,
			\[
			\dist(\widehat{D}_{\mu_{\bz}^n}(\bz),\widehat{D}^{0}(\bz)) \to 0 \qasq \sigma_{\bUpsilon} \to 0 ,
			\]
			where $\widehat{D}^{0}: V \to 2^{\conv{\supp(\mu)}}$ is the maximal monotone operator defined by
			\[
			\widehat{D}^{0}(\bz) = \conv{\projop^\Lambda_{\supp(\mu)}(\bz)} \subset \conv{\supp(\mu)} .
			\]
		\end{enumerate}
		\begin{proof}
			Upon passing to a subsequence, assume that $\lim\limits_{n\to \infty}\Lambda_n = \Lambda$ for a symmetric, positive-semidefinite matrix $\Lambda$ with $\norm{\Lambda} = 1$. Let $h_{\bz}(x) = \norm{\bx-\bz}_{\Lambda}$. By the compactness of $\supp(\mu)$, $h_{\bz}$, $h^2$ are uniformly continuous and $h(\bx)$ attains its minimium $d_\Lambda(\bz)$. Consequently, $\projop^\Lambda_{\supp(\mu)}(\bz)$ is non-empty and compact. For $\epsilon>0$ let \begin{equation*}
				\cU_{\bz}^\epsilon = \enscond{\bx \in \supp(\mu)}{h_{\bz}^2(\bx)\le d_\Lambda^2(\bz) +\epsilon}.
			\end{equation*}
			By uniform continuity, there exists $\delta_\epsilon >0$ such that \begin{equation*}
				\supp(\mu)\cap\Ball(\bx^\star,\delta_\epsilon) \subset \cU_{\bz}^\epsilon
			\end{equation*}
			for all $\bx^\star \in \projop^\Lambda_{\supp(\mu)}(\bz)$. Let \begin{equation*}
				\cK_{\bz}^\epsilon \eqdef \supp(\mu)\cap  \bigcup\limits_{\bx^\star \in \projop^\Lambda_{\supp(\mu)}(\bz)}\Ball(\bx^\star,\delta_\epsilon) \subset \cU_{\bz}^\epsilon.
			\end{equation*}
			Thus, by definition of the support and the separability of $\cX$, we have $\mu(\cK_{\bz}^\epsilon) > 0$. Further, the compactness of $\cX$ implies that $(\bx-\bz)^\top(\Lambda_n-\Lambda)(\bx-\bz)$ converges uniformly to $0$. Hence, for $n$ large enough it holds $\abs{(\bx-\bz)^\top(\Lambda_n-\Lambda)(\bx-\bz)} \le \frac{\epsilon}{4}$ for all $\bx\in \cX$. We can therefore estimate {\renewcommand{\jot}{12pt}\begin{align*}
					\mu_{\bz}^n(\supp(\mu)\setminus \cU_{\bz}^\epsilon) &= \frac{\int_{\supp(\mu)\setminus \cU_{\bz}^\epsilon} e^{-\frac{1}{2}(\bx-\bz)^\top\pa{\widetilde{\Sigma}_{\bUpsilon}^n}^+(\bx-\bz)}d\mu(\bx)}{\int_{\cX} e^{-\frac{1}{2}(\bx-\bz)^\top\pa{\widetilde{\Sigma}_{\bUpsilon}^n}^+(\bx-\bz)}d\mu(\bx)}\\
					&= \frac{\int_{\supp(\mu)\setminus \cU_{\bz}^\epsilon} e^{-\frac{1}{2}\norm{\pa{\widetilde{\Sigma}_{\bUpsilon}^n}^+}(\bx-\bz)^\top\Lambda_n(\bx-\bz)}d\mu(\bx)}{\int_{\cX} e^{-\frac{1}{2}\norm{\pa{\widetilde{\Sigma}_{\bUpsilon}^n}^+}(\bx-\bz)^\top\Lambda_n(\bx-\bz)}d\mu(\bx)}\\
					&= \frac{\int_{\supp(\mu)\setminus \cU_{\bz}^\epsilon} e^{-\frac{1}{2}\norm{\pa{\widetilde{\Sigma}_{\bUpsilon}^n}^+}(\bx-\bz)^\top(\Lambda_n-\Lambda)(\bx-\bz)} e^{-\frac{1}{2}\norm{\pa{\widetilde{\Sigma}_{\bUpsilon}^n}^+}(\bx-\bz)^\top\Lambda(\bx-\bz)}d\mu(\bx)}{\int_{\cX}e^{-\frac{1}{2}\norm{\pa{\widetilde{\Sigma}_{\bUpsilon}^n}^+}(\bx-\bz)^\top(\Lambda_n-\Lambda)(\bx-\bz)} e^{-\frac{1}{2}\norm{\pa{\widetilde{\Sigma}_{\bUpsilon}^n}^+}(\bx-\bz)^\top\Lambda(\bx-\bz)}d\mu(\bx)}\\
					&\le \frac{\int_{\supp(\mu)\setminus \cU_{\bz}^\epsilon} e^{\frac{\epsilon}{8}\norm{\pa{\widetilde{\Sigma}_{\bUpsilon}^n}^+}} e^{-\frac{1}{2}\norm{\pa{\widetilde{\Sigma}_{\bUpsilon}^n}^+}(\bx-\bz)^\top\Lambda(\bx-\bz)}d\mu(\bx)}{\int_{\cX}e^{-\frac{\epsilon}{8}\norm{\pa{\widetilde{\Sigma}_{\bUpsilon}^n}^+}} e^{-\frac{1}{2}\norm{\pa{\widetilde{\Sigma}_{\bUpsilon}^n}^+}(\bx-\bz)^\top\Lambda(\bx-\bz)}d\mu(\bx)}\\
					&= e^{\frac{\epsilon}{4}\norm{\pa{\widetilde{\Sigma}_{\bUpsilon}^n}^+}} \pa{\frac{\int_{\supp(\mu)\setminus \cU_{\bz}^\epsilon} e^{-\frac{1}{2} \norm{\pa{\widetilde{\Sigma}_{\bUpsilon}^n}^+} (h_{\bz}^2(x) - d^2_{\Lambda}(\bz)) }d\mu(\bx)}{\int_{\cX} e^{-\frac{1}{2} \norm{\pa{\widetilde{\Sigma}_{\bUpsilon}^n}^+} (h_{\bz}^2(x) - d^2_{\Lambda}(\bz)) }d\mu(\bx)}}\\
					&\le e^{\frac{\epsilon}{4}\norm{\pa{\widetilde{\Sigma}_{\bUpsilon}^n}^+}} \pa{\frac{\int_{\supp(\mu)\setminus \cU_{\bz}^\epsilon} e^{-\frac{1}{2} \norm{\pa{\widetilde{\Sigma}_{\bUpsilon}^n}^+} (h_{\bz}^2(x) - d^2_{\Lambda}(\bz)) }d\mu(\bx)}{\int_{\cK_{\bz}^{\epsilon/4}} e^{-\frac{1}{2} \norm{\pa{\widetilde{\Sigma}_{\bUpsilon}^n}^+} (h_{\bz}^2(x) - d^2_{\Lambda}(\bz)) }d\mu(\bx)}}\\
					&\le e^{\frac{\epsilon}{4}\norm{\pa{\widetilde{\Sigma}_{\bUpsilon}^n}^+}} \frac{e^{-\frac{3\epsilon}{8}\norm{\pa{\widetilde{\Sigma}_{\bUpsilon}^n}^+}}\mu(\supp(\mu)\setminus \cU_{\bz}^\epsilon)}{\mu(\cK_{\bz}^{\epsilon/4})} = \frac{e^{-\frac{\epsilon}{8}\norm{\pa{\widetilde{\Sigma}_{\bUpsilon}^n}^+}}\mu(\cX)}{\mu(\cK_{\bz}^{\epsilon/4})} .
			\end{align*}}
			Since $\lim\limits_{n\to\infty} \widetilde{\Sigma}_{\bUpsilon}^n = 0$, we have $\lim\limits_{n\to\infty}\norm{\pa{\widetilde{\Sigma}_{\bUpsilon}^n}^+} = \infty$ and thus $\lim\limits_{n\to\infty} \mu_{\bz}^n(\supp(\mu)\setminus \cU_{\bz}^\epsilon) = 0$. From here we continue as in the proof of Proposition~\ref{prop:limsigma}.
		\end{proof}
	\end{proposition}
	

	\begin{remark}
		Note that even if $\projop^{\Lambda_n}_{\supp(\mu)}(\bz)$ converges as a sequence of set, its limit man not coincide with $\projop^\Lambda_{\supp(\mu)}(\bz)$, e.g. if $\rank{\Lambda} < \liminf\limits_{n\to \infty}\rank{\Lambda_n}$. 
	\end{remark}

	\bibliographystyle{plain}
	\bibliography{sample}
\end{document}